\documentclass[12pt]{amsart}
\usepackage{amsmath,fullpage}
\usepackage{amscd}
\usepackage{amssymb}
\usepackage[usenames]{color}

\usepackage[all,cmtip]{xy}

\newtheorem{theorem}{Theorem}[section]
\newtheorem{corollary}[theorem]{Corollary}
\newtheorem{lemma}[theorem]{Lemma}
\newtheorem{proposition}[theorem]{Proposition}

\theoremstyle{definition}
\newtheorem{definition}[theorem]{Definition}
\newtheorem{remark}[theorem]{Remark}

\newcommand{\cB}{{\mathcal B}}

\newcommand{\cO}{{\mathcal O}}

\newcommand{\cF}{{\mathcal F}}

\newcommand{\cH}{{\mathcal H}}
\newcommand{\cT}{{\mathcal T}}

\newcommand{\bQ}{{\mathbb Q}}

\newcommand{\tr}{{\text{tr}}}

\newcommand{\Hom}{{\text{Hom}}}

\newcommand{\End}{{\text{End}}}
\newcommand{\Ad}{{\text{Ad}}}

\newcommand{\supp}{{\text{supp}}}

\newcommand{\rc}{{\mathrm c}}

\newcommand{\Lb}{{\mathfrak{b}}}
\newcommand{\Lt}{{\mathfrak{t}}}
\newcommand{\Lg}{{\mathfrak g}}
\newcommand{\Lo}{{\mathfrak o}}
\newcommand{\Ln}{{\mathfrak{n}}}

\newcommand{\A}{{\textbf{A}}}
\newcommand{\tF}{{\textbf{F}}}
\newcommand{\tk}{{\textbf{k}}}

\newcommand{\p}{\perp}

\newcommand{\beq}{\begin{equation*}}
\newcommand{\eeq}{\end{equation*}}

\begin{document}
\title[Nilpotent orbits in the dual of classical Lie algebras ]{Nilpotent orbits
in the dual of classical Lie algebras in characteristic 2 and the
Springer correspondence}
        \author{Ting Xue}
        \address{Department of Mathematics, Massachusetts Institute of Technology,
Cambridge, MA 02139, USA}
        \email{txue@math.mit.edu}
\maketitle

\begin{abstract}
Let  $G$ be a simply connected algebraic group of type $B,C$ or $D$
over an algebraically closed field of characteristic 2. We construct
a Springer correspondence for the dual vector space of the Lie
algebra of $G$. In particular, we classify the nilpotent orbits in
the duals of symplectic  and orthogonal Lie algebras over
algebraically closed or  finite fields of characteristic 2.
\end{abstract}
\section{Introduction}
Throughout this paper, let $\tk$ be a field of characteristic 2. Let
$G$ be an algebraic group of type $B,C$ or $D$ over $\tk$ and $\Lg$
be its Lie algebra. Let $\Lg^*$ be the dual vector space of $\Lg$.
We have a natural action of $G$ on $\Lg^*$,
$g.\xi(x)=\xi(\Ad(g)^{-1}x)$ for $g\in G,\xi\in\Lg^*$ and $x\in\Lg$.
Fix a Borel subgroup $B$ of $G$ and let $\Lb$ be the Lie algebra of
$B$. Let $\Ln'=\{\xi\in\Lg^*|\xi(\Lb)=0\}$. An element $\xi$ in
$\Lg^*$ is called nilpotent if there exists $g\in G$ such that
$g.\xi\in\Ln'$( see \cite{KW}). We classify the nilpotent orbits in
$\Lg^*$ under the action of $G$ in the cases where $\tk$ is
algebraically closed and where $\tk$ is a finite field $\tF_q$. In
particular, we obtain the number of nilpotent orbits over $\tF_q$
and the structure of component groups of the centralizers of
nilpotent elements.

Let $G_s$ be a simply connected algebraic group of type $B,C$ or $D$
defined over $\tk$ (assume $\tk$ algebraically closed) and $\Lg_s$
be the Lie algebra of $G_s$. Let $\Lg_{s}^*$ be the dual vector
space of $\Lg_{s}$. Let $\mathfrak{A}'_{s}$ be the set of all pairs
$(\mathrm{c}',\mathcal{F}')$ where $\mathrm{c}'$ is a nilpotent
$G_{s}$-orbit in $\Lg_{s}^*$ and $\mathcal{F}'$ is an irreducible
$G_{s}$-equivariant local system on $\mathrm{c}'$ (up to
isomorphism). We construct a Springer correspondence for $\Lg_{s}^*$
using a similar construction as in \cite{Lu1,Lu4,X}. The
correspondence is a bijective map from the set of isomorphism
classes of irreducible representations of the Weyl group of $G_{s}$
to the set $\mathfrak{A}_{s}'$.

\section{symplectic groups}
In this section we study the nilpotent orbits in $\Lg^*$ where $G$
is a symplectic group.
\subsection{}\label{ssec-1-1}
Let $V$ be a vector space of dimension $2n$ over $\tk$ equipped with
a non-degenerate symplectic form $\beta:V\times V\rightarrow \tk$.
The symplectic group is defined as $G=Sp(2n)=\{g\in GL(V)\ |\
\beta(gv,gw)=\beta(v,w), \forall\ v,w\in V\}$ and its Lie algebra is
$\Lg=\mathfrak{sp}(2n)=\{x\in \mathfrak{gl}(V)\ |\
\beta(xv,w)+\beta(v,xw)=0, \forall\ v,w\in V\}$.

Let $\xi$ be an element of $\Lg^*$. There exists $X\in
\mathfrak{gl}(V)$ such that $\xi(x)=\tr(Xx)$ for any $x\in\Lg$. We
define a quadratic form $\alpha_\xi:V\rightarrow \tk$ by
$$\alpha_\xi(v)=\beta(v,Xv).$$
\begin{lemma}\label{lem-7}
The quadratic form $\alpha_\xi$ is well-defined.
\end{lemma}
\begin{proof} Recall that the space $\text{Quad}(V)$ of
quadratic forms on $V$ coincides with the second symmetric power
$S^2(V^*)$ of $V^*$. Consider the following linear mapping \beq
\Phi:\End_\tk(V)\rightarrow S^2(V^*)=\text{Quad}(V),\quad
X\mapsto\alpha_X\eeq where $\alpha_X(v)=\beta(v,Xv)$. It is easy to
see that $\Phi$ is $G=Sp(V)$-equivariant. One can show that
$\ker\Phi$ coincides with the orthogonal complement $\Lg^\p$ of
$\Lg=\mathfrak{sp}(V)$ in $\End_\tk(V)$ under the nondegenerate
trace form. It follows that $\alpha_\xi$ does not depend on the
choice of $X$.
\end{proof}
\begin{remark}
I thank the referee for suggesting the present coordinate free
proofs of Lemmas \ref{lem-7},  \ref{lem-nilp1}, \ref{lem-3-1},
\ref{lem-n-3}, \ref{lem-vu}, \ref{lem-5} and \ref{lem-w3}. These
proofs replace my earlier proofs for which coordinates are used.
\end{remark}

Let $\beta_\xi$ be the symmetric bilinear form associated to
$\alpha_\xi$, namely,
$\beta_\xi(v,w)=\alpha_\xi(v+w)+\alpha_\xi(v)+\alpha_\xi(w)$,
$v,w\in V$. Define a linear map $T_\xi:V\rightarrow V$ by
$$\beta(T_\xi v,w)=\beta_\xi(v,w).$$
Assume $\xi\in\Lg^*$. We denote by $(V_\xi,\beta,\alpha_\xi)$ the
vector space $V$ equipped with the symplectic form $\beta$ and the
quadratic form $\alpha_\xi$.
\begin{definition}
Assume $\xi,\zeta\in\Lg^*$. We say that $(V_\xi,\beta,\alpha_\xi)$
is equivalent to $(V_\zeta,\beta,\alpha_\zeta)$ if there exists a
vector space isomorphism $g:V_\xi\rightarrow V_\zeta$ such that
$\beta(gv,gw)=\beta(v,w)$ and $\alpha_\zeta(gv)=\alpha_\xi(v)$ for
all $v,w\in V_\xi$.
\end{definition}

\begin{lemma}
Two elements $\xi,\zeta\in\Lg^*$ lie in the same $G$-orbit if and
only if there exists $g\in G$ such that
$\alpha_\xi(g^{-1}v)=\alpha_\zeta(v),\  \forall\  v\in V$.
\end{lemma}
\begin{proof} The two elements $\xi,\zeta$ lie in the
same $G$-orbit if and only if there exists $g\in G$ such that
$g.\xi(x)=\xi(g^{-1}xg)=\zeta(x),\ \forall\ x\in \Lg$. Assume
$\xi(x)=\tr(X_\xi x)$ and $\zeta(x)=\tr(X_\zeta x)$. Similar
argument as in the proof of Lemma  \ref{lem-7} shows that
$g.\xi(x)=\tr(gX_\xi g^{-1}x)=\zeta(x)$ if and only if $\beta(gX_\xi
g^{-1}v,v)+\beta(X_\zeta v,v)=0$ if and only if
$\alpha_\xi(g^{-1}v)=\alpha_\zeta(v),\ \forall\ v\in V$.
\end{proof}

\begin{corollary}\label{cor-2}
Two elements $\xi,\zeta\in\Lg^*$ lie in the same $G$-orbit if and
only if $(V_\xi,\beta,\alpha_\xi)$ is equivalent to
$(V_\zeta,\beta,\alpha_\zeta)$.
\end{corollary}
\subsection{}\label{ssec-1-2}
From now on we assume that $\xi\in\Lg^*$ is nilpotent.
\begin{lemma}\label{lem-nilp1}
Let $\xi\in\Lg^*$ be nilpotent. Then $T_\xi$ is a nilpotent element
in $\mathfrak{gl}(V_\xi)$.
\end{lemma}

\begin{proof} Note that the nilpotent elements in
$\Lg^*$ (resp. $\Lg$) are precisely the "unstable" vectors $\xi$
(resp. $x$), namely, those $\xi$ (resp. $x$) for which the closure
of the $G$-orbit $G.\xi$ (resp. $\Ad(G)x$) contains $0$. By
Hilbert's criterion for instability, there exists a co-character
$\phi:\textbf{G}_m\rightarrow G$ such that $\lim_{a\rightarrow
0}\phi(a).\xi=0$. To show that $T_\xi$ is nilpotent, it is enough to
show that $\lim_{a\rightarrow 0}\Ad(\phi(a))T_\xi=0$.

For any $G$-representation $M$ and $i\in\mathbb{Z}$, we write
$M(\phi;i)$ for the $i$-weight space of the torus
$\{\phi(a)\}_{a\in\textbf{G}_m}$ and
$M(\phi;>i)=\oplus_{j>i}M(\phi;j)$, and similarly for $M(\phi;\geq
i)$, $M(\phi;\leq i)$ etc.

Since $\xi\in\Lg^*(\phi,>0)$, we may choose
$X\in\End_\tk(V)(\phi;>0)$ such that $\xi(x)=\tr(Xx)$ for all
$x\in\Lg$. Notice that we have
\begin{eqnarray*}
&&\beta((\Ad(\phi(a))T_\xi) v,w)=\beta(T_\xi
\phi(a)^{-1}v,\phi(a)^{-1}w)=\beta_\xi(\phi(a)^{-1}v,\phi(a)^{-1}w)
\\&&=\beta(X\phi(a)^{-1}v,\phi(a)^{-1}w)+\beta(\phi(a)^{-1}v,X\phi(a)^{-1}w)\\&&=
\beta((\Ad(\phi(a))X)v,w)+\beta(v,(\Ad(\phi(a))X)w).
\end{eqnarray*}
Since $X\in\End_\tk(V)(\phi;>0)$, $\Ad(\phi(a))X\rightarrow 0$ as
$a\rightarrow 0$ and thus $\beta(\Ad(\phi(a))T_\xi v,w)\rightarrow
0$ as $a\rightarrow 0$ for any $v,w\in V$. It follows that
$\Ad(\phi(a))T_\xi\rightarrow 0$ as $a\rightarrow 0$, since the
bilinear form $\beta$ is nondegenerate. Thus $T_\xi$ is nilpotent.
\end{proof}

Let $A=\tk[[t]]$ be the ring of formal power series in the
indeterminate $t$.  We consider $V_\xi$ as an $A$-module by $(\sum
a_kt^k)v=\sum a_kT_\xi^kv$. Let $E$ be the vector space spanned by
the linear functionals $t^{-k}: A\rightarrow \tk,\ \sum
a_it^i\mapsto a_k,\ k\geq 0$. Let $E_0$ and $E_1$ be the subspace
$\sum \tk t^{-2k}$ and $\sum \tk t^{-2k-1}$ respectively. Denote
$\pi_i:E\rightarrow E_i$, $i=0,1$ the natural projections. The
vector space $E$ is considered as an $A$-module by $ (au)(b)=u(ab)$
for $a,b\in A,u\in E$. Define $\varphi:V\times V\rightarrow
E,\psi:V\rightarrow E_1$ and $\varphi_\xi:V\times V\rightarrow
E,\psi_\xi:V\rightarrow E_0$ by
$$ \varphi(v,w)=\sum_{k\geq 0}\beta(t^kv,w) t^{-k},\ \psi(v)=\sum_{k\geq 0}
\beta(t^{k+1}v,t^k v)t^{-2k-1}$$ and $$ \varphi_\xi(v,w)=\sum_{k\geq
0}\beta_\xi(t^kv,w) t^{-k},\ \psi_\xi(v)=\sum_{k\geq 0}
\alpha_\xi(t^k v)t^{-2k}.$$ Notice that we have $\beta(T_\xi
v,v)=\beta_\xi(v,v)=0$ and $\beta_\xi(T_\xi v, v)=\beta(T_\xi
v,T_\xi v)=0$. By Proposition 2.7 in \cite{Hes}, we can identify
$(V_\xi,\alpha=0,\beta)$ with $(V_\xi,\varphi,\psi)$,
$(V_\xi,\alpha_\xi,\beta_\xi)$ with $(V_\xi,\varphi_\xi,\psi_\xi)$
and hence $(V_\xi,\beta,\alpha_\xi)$ with
$(V_\xi,\varphi,\psi,\varphi_\xi,\psi_\xi)$. The mappings
$\varphi,\psi$ and $\varphi_\xi,\psi_\xi$ satisfy the following
properties (\cite{Hes}):
\begin{enumerate}
  \item[(i)] The maps $\varphi(\cdot,w)$ and $\varphi_\xi(\cdot,w)$ are $A$-linear for every $w\in V_\xi$.
  \item[(ii)] $\varphi(v,w)=\varphi(w,v)$, $\varphi_\xi(v,w)=\varphi_\xi(w,v)$ for all $v,w\in V_\xi.$
  \item[(iii)] $\varphi(v,v)=\psi(v)$, $\varphi_\xi(v,v)=0$ for all $v\in V_\xi$.
  \item[(vi)] $\psi(v+w)=\psi(v)+\psi(w)$,
  $\psi_\xi(v+w)=\psi_\xi(v)+\psi_\xi(w)+\pi_0(\varphi_\xi(v,w))$ for all $v,w\in V_\xi.$
  \item[(v)] $\psi(a v)=a^2\psi(v)$,$\psi_\xi(a v)=a^2\psi_\xi(v)$ for all $v\in V_\xi,\ a\in A$.
\end{enumerate}

Following \cite{Hes}, we call $(V_\xi,\beta,\alpha_\xi)$ a form
module and $(V_\xi,\varphi,\psi,\varphi_\xi,\psi_\xi)$ an abstract
form module. Corollary \ref{cor-2} says that classifying the
nilpotent $G$-orbits in $\Lg^*$ is equivalent to classifying the
equivalence classes of the form modules $(V_\xi,\beta,\alpha_\xi)$.
In the following we classify the form modules
$(V_\xi,\beta,\alpha_\xi)$ via the identification with
$(V_\xi,\varphi,\psi,\varphi_\xi,\psi_\xi)$. We write
$V_\xi=(V_\xi,\beta,\alpha_\xi)$.

Since $T_\xi$ is nilpotent, there exists a unique sequence of
integers $p_1\geq\cdots\geq p_s\geq 1$ and a family of vectors
$v_1,\ldots,v_s$ such that $T_\xi^{p_i}v_i=0$ and the vectors
$T_\xi^{q_i}v_i$, $0\leq q_i\leq p_i-1$ form a basis of $V$. We
define $p(V_\xi)=p(T_\xi)=(p_1,\ldots,p_s)$. Define an index
function $\chi_{V_\xi}:\mathbb{Z}\rightarrow \mathbb{N}$ for
$(V_\xi,\beta,\alpha_\xi)$
 by
$$\chi_{V_\xi}(m)=\text{min}\{i\geq 0|T_\xi^mv=0\Rightarrow
\alpha_\xi(T_\xi^iv)=0\}.$$ Define $\mu(V_\xi)$ to be the minimal
integer $m\geq 0$ such that $T_\xi^m V_\xi=0$. For $v\in V_\xi$, we
define $\mu(v)=\mu(Av)$. We define $\mu(E)$ for $E$ and $\mu(u)$ for
$u\in E$ similarly.

\begin{lemma}\label{lem-2}
We have $\psi(v)=0$ and $\varphi_\xi(v,w)=t\varphi(v,w)$ for all
$v,w\in V_\xi$.
\end{lemma}
\begin{proof}
The first assertion follows from $\beta(T_\xi v,v)=0$, $\forall\
v\in V_\xi$. The second assertion follows from
$\beta_\xi(T_\xi^kv,w)=\beta(T_\xi^{k+1}v,w)$.
\end{proof}

 We study the orthogonal decomposition of
$V_\xi$ with respect to $\varphi$, which is also an orthogonal
decomposition of $V_\xi$ with respect to $\varphi_\xi$ since
$\varphi(v,w)=0$ implies $ \varphi_\xi(v,w)=0$ (Lemma \ref{lem-2}).
Recall that an orthogonal decomposition of $V$ is an expression of
$V$ as a direct sum $V=\sum_{i=1}^r V_i$ of mutually orthogonal
submodules $V_i$. A form module $V$ is called indecomposable if
$V\neq 0$ and for every orthogonal decomposition $V=V_1\oplus V_2$
we have $V_1=0$ or $V_2=0$. Every form module $V$ has some
orthogonal decomposition $V=\sum_{i=1}^r V_i$ in indecomposable
submodules $V_1,V_2,\ldots,V_r$.

We first classify the indecomposable modules (with respect to
$\varphi$) that appear in the orthogonal decompositions of form
modules $(V_\xi,\beta,\alpha_\xi)$. Let
$(V_\xi,\varphi,\psi,\varphi_\xi,\psi_\xi)$ be an indecomposable
module, where $\xi\in\Lg^*$ is nilpotent. Since $\psi(v)=0$ for all
$v\in V_\xi$ (Lemma \ref{lem-2}), by the classification of modules
$(V_\xi,\varphi,\psi)$, there exist $v_1,v_2$ such that
$V_\xi=Av_1\oplus Av_2$ with $\mu(v_1)=\mu(v_2)=m$ and
$\varphi(v_1,v_2)=t^{1-m}$ (see \cite{Hes} section 3.5, notice that
$\beta$ is non-degenerate on $V_\xi$). Denote
$\psi_\xi(v_1)=\Psi_1$, $\psi_\xi(v_2)=\Psi_2$ and
$\varphi_\xi(v_1,v_2)=t^{2-m}=\Phi_\xi$.
\subsection{}
In this subsection assume $\tk$ is algebraically closed.
\begin{proposition}\label{prop-1.1}
The indecomposable modules are $^*W_l(m)=Av_1\oplus Av_2$,
$[\frac{m}{2}]\leq l\leq m$, with $\mu(v_1)=\mu(v_2)=m$,
$\psi_\xi(v_1)=t^{2-2l}$, $\psi_\xi(v_2)=0$ and
$\varphi(v_1,v_2)=t^{1-m}$. We have $\chi_{^*W_l(m)}=[m;l]$, where
$[m;l]:\mathbb{N}\rightarrow\mathbb{Z}$ is defined by
$[m;l](k)=\max\{0,\min\{k-m+l,l\}\}.$
\end{proposition}
\begin{proof}
Assume $\mu(\Psi_1)\geq\mu(\Psi_2)$. Let $v_2'=v_2+av_1$. The
equation $\psi_\xi(v_2')=\Psi_2+a^2\Psi_1+\pi_0(a\Phi_\xi)=0$ has a
solution for $a$, hence we can assume $\Psi_2=0$. Assume
$\Psi_1=\sum_{i=0}^{l}a_it^{-2i}$, $a_i\in \tk, a_l\neq 0$. Let
$v_1'=av_1$, $a\in A$. We can take $a$ invertible in $A$ such that
$\psi_\xi(v_1')=t^{-2l}$. Let $v_2''=a^{-1}v_2'$. One verifies that
$\psi_\xi(v_1')=t^{-2l}$, $\psi_\xi(v_2'')=0$ and
$\varphi(v_1',v_2'')=t^{1-m}$. Furthermore, we can assume
$[m/2]-1\leq l\leq m-1$. In fact, we have $l\leq m-1$ since
$t^mv=0,\forall\  v\in V$; if $l<[\frac{m}{2}]-1$, let
$v_1'=v_1+t^{m-2l-2}v_2+t^{m-2[\frac{m}{2}]}v_2$, then
$\psi_\xi(v_1')=t^{-2([\frac{m}{2}]-1)}$ and
$\varphi(v_1',v_2)=t^{1-m}$. One can verify that the modules
$^*W_l(m)$, $[m/2]\leq l\leq m$ exist and are not equivalent to each
other. 
\end{proof}

\begin{lemma}\label{lem-1}
Assume $m_1\geq m_2$.

$\mathrm{(i)}$ If $l_1<l_2$, we have $^*W_{l_1}(m_1)\oplus
{^*W_{l_2}}(m_2)\cong {^*W_{l_2}}(m_1)\oplus {^*W_{l_2}}(m_2)$.

$\mathrm{(ii)}$ If $m_1-l_1<m_2-l_2$, we have $^*W_{l_1}(m_1)\oplus
{^*W_{l_2}}(m_2)\cong {^*W_{l_1}}(m_1)\oplus
{^*W_{m_2-m_1+l_1}}(m_2)$.
\end{lemma}
\begin{proof}
Assume $^*W_{l_1}(m_1)\oplus {^*W_{l_2}}(m_2)=Av_1\oplus Aw_1\oplus
Av_2\oplus Aw_2$ with $\psi_{\xi}(v_i)=t^{2-2l_i},\psi_{\xi}(w_i)=0$
and
$\varphi(v_i,w_j)=\delta_{i,j}t^{1-m_i},\varphi(v_i,v_j)=\varphi(w_i,w_j)=0$,
$i,j=1,2$.  Let $\tilde{v}_1=v_1+(1+t^{l_2-l_1})v_2$,
$\tilde{w}_1=w_1$, $\tilde{v}_2=v_2$,
$\tilde{w_2}=w_2+(t^{m_1-m_2}+t^{m_1-l_1-m_2+l_2})w_1$. Then we have
$\psi_{\xi}(\tilde{v}_i)=t^{2-2l_2}$, $\psi_{\xi}(\tilde{w}_i)=0$
and $\varphi(\tilde{v}_i,\tilde{w}_j)=\delta_{i,j}t^{1-m_i},
\varphi(\tilde{v}_i,\tilde{v}_j)=\varphi(\tilde{w}_i,\tilde{w}_j)=0$,
$i,j=1,2$. This proves (i). One can prove (ii) similarly.
\end{proof}

\begin{remark}Notice that we do not have a
"Krull-Schmidt" type theorem here, namely, the indecomposable
summands of a form module $V$ are not uniquely determined by $V$.
(See also Lemma \ref{lem-2-1} and Lemma \ref{lem-6}.)
\end{remark}

By Proposition \ref{prop-1.1} and Lemma \ref{lem-1}, for every
module $V$, there exists a unique sequence of modules
$^*W_{l_i}(m_i)$ such that $V$ is equivalent to
$^*W_{l_1}(m_1)\oplus {^*W_{l_2}}(m_2)\oplus\cdots\oplus
{^*W_{l_s}}(m_s)$, $[\frac{m_i}{2}]\leq l_i\leq m_i$, $m_1\geq
m_2\geq\cdots\geq m_s$, $l_1\geq l_2\geq\cdots\geq l_s$ and
$m_1-l_1\geq m_2-l_2\geq\cdots\geq m_s-l_s$. Thus the equivalence
class of $V$ is characterized by the symbol
$$(m_1)^2_{l_1}\cdots(m_s)^2_{l_s}.$$ A symbol of the above form is
the symbol of a form module if and only if $[\frac{m_i}{2}]\leq
l_i\leq m_i$, $m_i\geq m_{i+1}$, $l_i\geq l_{i+1}$ and $m_i-l_i\geq
m_{i+1}-l_{i+1}$, $i=1,\ldots,s$.

It follows that $p(V_\xi)=m_1^2\cdots m_s^2$ and
$\chi_{V_\xi}(k)=\text{sup}_i\chi_{^*W_{l_i}(m_i)}(k)$ for all
$k\in\mathbb{N}$ and $\chi_V(m_i)=\chi_{^*W_{l_i}(m_i)}=l_i$. Thus
we have the following proposition.
\begin{proposition}
Two nilpotent elements $\xi,\zeta\in\Lg^*$ lie in the same $G$-orbit
if and only if $T_\xi,T_\zeta$ are conjugate in $GL(V)$ and
$\chi(V_\xi)=\chi(V_\zeta)$.
\end{proposition}

We associate to the orbit $(m_1)^2_{l_1}\cdots(m_s)^2_{l_s}$ a pair
of partitions $(l_1,\ldots,l_s)(m_1-l_1,\ldots,m_s-l_s).$ In this
way we construct a bijection from the set of nilpotent orbits in
$\Lg^*$ to the set $\{(\mu,\nu)||\mu|+|\nu|=n,\nu_i\leq \mu_i+1\}$,
which has cardinality $p_2(n)-p_2(n-2)$. Here and afterwards we
denote by $p_2(n)$ the number of pairs of partitions $(\mu,\nu)$
such that $|\mu|+|\nu|=n$.

\subsection{}

In this subsection, let $\tk=\tF_q$. Let $G(\tF_q)$,
$\mathfrak{g}({\tF}_q)$ be the fixed points of a Frobenius map
$\mathfrak{F}_q$ relative to $\tF_q$ on $G$, $\Lg$. We study the
nilpotent $G(\tF_q)$-orbits in $\mathfrak{g}({\tF}_q)^*$. Fix
$$\delta\notin \{x^2+x|x\in\tF_q\}.$$ We have the following statements
whose proofs are entirely similar to those of \cite{X}. For
completeness, we also include the proofs here.

\begin{proposition}\label{prop-nind}
The indecomposable modules over $\tF_q$ are

$\mathrm{(i)}$ $^*W_l^0(m)=Av_1\oplus Av_2$, $(m-1)/2\leq l\leq m$
with $\psi_\xi(v_1)=t^{2-2l}$,  $\psi_\xi(v_2)=0$ and
$\varphi(v_1,v_2)=t^{1-m}$;

$\mathrm{(ii)}$ $^*W_l^\delta(m)=Av_1\oplus Av_2$, $(m-1)/2< l< m$
with $\psi_\xi(v_1)=t^{2-2l}$,  $\psi_\xi(v_2)=\delta t^{-2(m-1-l)}$
and $\varphi(v_1,v_2)=t^{1-m}$.
\end{proposition}
\begin{proof}
Let $V_\xi=Av_1\oplus Av_2$ be an indecomposable module as in the
last paragraph of subsection \ref{ssec-1-2}. We have
$\Phi_\xi=t^{2-m}$. We can assume that $\mu(\Psi_1)\geq\mu(\Psi_2)$.
We have the following cases:

Case 1: $\Psi_1=\Psi_2=0$. Let
$\tilde{v}_1=v_1+t^{m-2[\frac{m}{2}]}v_2,\tilde{v}_2=v_2$, then we
have $\psi_\xi(\tilde{v}_1)=t^{2-2[\frac{m}{2}]},\
\psi_\xi(\tilde{v}_2)=0$ and
$\varphi(\tilde{v}_1,\tilde{v}_2)=t^{1-m}$.

Case 2: $\Psi_1\neq 0,\ \Psi_2=0$. There exist $a,b\in A$
invertible, such that $\psi_\xi(av_1)=t^{-2k},\psi_\xi(bv_2)=0$ and
$\varphi(av_1,bv_2)=t^{1-m}$. Hence we can assume $\Psi_1=t^{-2k}$
where $k\leq m-1$. If $k<[\frac{m}{2}]-1$, let
$\tilde{v}_1=v_1+t^{m-2[\frac{m}{2}]}v_2+t^{m-2k-2}v_2,\
\tilde{v}_2=v_2$; otherwise, let $\tilde{v}_1=v_1,\
\tilde{v}_2=v_2$. Then we get $\psi_\xi(\tilde{v}_1)=t^{-2k},\
[\frac{m}{2}]-1\leq k\leq m-1,\ \psi_\xi(\tilde{v}_2)=0,
         \ \varphi(\tilde{v}_1,\tilde{v}_2)=t^{1-m}$.

Case 3: $\Psi_1\neq 0,\ \Psi_2\neq 0$. There exist $a,b\in A$
invertible, such that $\psi_\xi(av_1)=t^{-2l_1}$ and
$\varphi(av_1,bv_2)=t^{1-m}$. Hence we can assume $\Psi_1=t^{-2l_1}$
and $\Psi_2=\sum_{i=0}^{l_2}a_i t^{-2i}$ where $l_2\leq l_1\leq
m-1$. Let $\tilde{v}_2=v_2+\sum_{i=0}^{m-1}x_it^iv_1$. Assume
$l_1<\frac{m-2}{2}$, then $\psi_\xi(\tilde{v}_2)=0$ has a solution
for $x_i$'s and we get Case 2. Assume $l_1\geq \frac{m-2}{2}$. If
$a_{m-l_1-2}\in\{x^2+x|x\in \tF_q\}$, then $\psi_\xi(\tilde{v}_2)=0$
has a solution for $x_i$'s and we get Case 2; if
$a_{m-l_1-2}\notin\{x^2+x|x\in \tF_q\}$, then
$\psi_\xi(\tilde{v}_2)=\delta t^{-2(m-l_1-2)}$ has a solution for
$x_i$'s.

Summarizing Cases 1-3, we have normalized $V_\xi=Av_1\oplus Av_2$
with $\mu(v_1)=\mu(v_2)=m$ as follows:

(i) $(m-1)/2\leq\chi(m)=l\leq m$, $\psi_\xi(v_1)=t^{2-2l},\
\psi_\xi(v_2)=0,\
             \varphi(v_1,v_2)=t^{1-m}$, denoted by $^*W_l^0(m)$.

(ii) $(m-1)/2<\chi(m)=l< m$, $\psi_\xi(v_1)=t^{2-2l},\
\psi_\xi(v_2)=\delta t^{-2(m-l-1)},\ \varphi(v_1,v_2)=t^{1-m}$,
denoted by $^*W_l^{\delta}(m)$.

We show that $^*W_l^0(m)$ and $^*W_l^\delta(m)$, where
$\frac{m-1}{2}<l<m$, are not equivalent. Take $v_i,w_i$, $i=1,2$,
such that $^*W_l^0(m)=Av_1\oplus Aw_1$, $^*W_l^\delta(m)=Av_2\oplus
Aw_2$, $\mu(v_i)=\mu(w_i)=m$,
$\psi_\xi(v_i)=t^{2-2l},\psi_\xi(w_1)=0,\psi_\xi(w_2)=\delta
t^{2l-2m+2}$ and $\varphi(v_i,w_i)=t^{1-m}$, $i=1,2$. The modules
$^*W_l^0(m)$ and $^*W_l^\delta(m)$ are equivalent if and only if
there exists a linear isomorphism $g:{^*W}_l^0(m)\rightarrow
{^*W}_l^\delta(m)$ such that $\psi_\xi(gv)=\psi_\xi(v)$ and
$\varphi(gv,gw)=\varphi(v,w)$ for all $v,w\in {^*W}_l^0(m)$. Assume
$gv_1=\sum_{i=0}^{m-1}(a_it^{i}v_2+b_it^iw_2),gw_1=\sum_{i=0}^{m-1}(c_it^{i}v_2+d_it^iw_2)$.
Then a straightforward calculation shows that if $l=\frac{m}{2}$,
among the equations
$\psi_\xi(gv_1)=\psi_\xi(v_1),\psi_\xi(gw_1)=\psi_\xi(w_1),\varphi(gv_1,gw_1)=\varphi(v_1,w_1)$,
the following equations appear
\begin{eqnarray*}
&&c_0^2+\delta d_0^2+c_0d_0=0\\
&&a_0d_0+b_0c_0=1.
\end{eqnarray*}
It follows that $c_0,d_0\neq 0$ and thus the first equation becomes
an "Artin-Schreier" equation
$(\frac{c_0}{d_0})^2+\frac{c_0}{d_0}=\delta$ which has no solutions
over $\tF_q$. Similarly if $\frac{m}{2}<l<m$, an "Artin-Schreier"
equation $c_{2l-m}^2+c_{2l-m}=\delta$ appears. It follows that
$^*W_l^0(m)$ and $^*W_l^\delta(m)$, where $\frac{m-1}{2}<l<m$, are
not equivalent.
\end{proof}

\begin{remark}\label{rmk-1}
It follows that the equivalence class of the form module $^*W_l(m)$
over $\bar{\tF}_q$ remains as one equivalence class over $\tF_q$
when $l=\frac{m-1}{2}$ or $l=m$ and decomposes into two equivalence
classes $^*W_l^0(m)$ and $^*W_l^\delta(m)$ over $\tF_q$ otherwise.
\end{remark}

\begin{lemma}\label{lem-2-1}
Assume $l_1\geq l_2$ and $m_1-l_1\geq m_2-l_2$.

$\mathrm{(i)}$  If $l_1+l_2<m_1$, we have that
${^*W}_{l_1}^0(m_1)\oplus {^*W}_{l_2}^0(m_2)$,
${^*W}_{l_1}^0(m_1)\oplus {^*W}_{l_2}^\delta(m_2)$,
${^*W}_{l_1}^\delta(m_1)\oplus {^*W}_{l_2}^0(m_2)$ and
${^*W}_{l_1}^\delta(m_1)\oplus {^*W}_{l_2}^\delta(m_2)$ are not
equivalent to each other.

$\mathrm{(ii)}$  If $l_1+l_2\geq m_1$, we have
${^*W}_{l_1}^0(m_1)\oplus {^*W}_{l_2}^0(m_2)\cong
{^*W}_{l_1}^\delta(m_1)\oplus {^*W}_{l_2}^\delta(m_2)$ and
${^*W}_{l_1}^0(m_1)\oplus {^*W}_{l_2}^\delta(m_2)\cong
{^*W}_{l_1}^\delta(m_1)\oplus {^*W}_{l_2}^0(m_2)$. The two pairs are
not equivalent to each other.
\end{lemma}
\begin{proof}
We show that ${^*W}_{l_1}^0(m_1)\oplus {^*W}_{l_2}^\delta(m_2)$ and
${^*W}_{l_1}^\delta(m_1)\oplus {^*W}_{l_2}^0(m_2)$ are equivalent if
and only if $l_1+l_2\geq m_1$. The other statements are proved
similarly. Assume ${^*W}_{l_1}^0(m_1)\oplus {^*W}_{l_2}^\delta(m_2)$
and ${^*W}_{l_1}^\delta(m_1)\oplus {^*W}_{l_2}^0(m_2)$ correspond to
$\xi$ and $\xi'$ respectively. Take $v_1,w_1$ and $v_2,w_2$ such
that ${^*W}_{l_1}^0(m_1)\oplus {^*W}_{l_2}^\delta(m_2)=Av_1\oplus
Aw_1\oplus Av_2\oplus Aw_2$ and
$\psi_\xi(v_1)=t^{2-2l_1},\psi_\xi(w_1)=0,\varphi(v_1,w_1)=t^{1-m_1},
\psi_\xi(v_2)=t^{2-2l_2},\psi_\xi(w_2)=\delta t^{-2(m_2-l_2-1)},
\varphi(v_2,w_2)=t^{1-m_2},\varphi(v_1,v_2)=\varphi(v_1,w_2)
=\varphi(w_1,v_2)=\varphi(w_1,w_2)=0. $ Similarly, take $v_1',w_1'$
and $v_2',w_2'$ such that ${^*W}_{l_1}^\delta(m_1)\oplus
{^*W}_{l_2}^0(m_2)=Av_1'\oplus Aw_1'\oplus Av_2'\oplus Aw_2'$ and
$\psi_{\xi'}(v_1')=t^{2-2l_1},\psi_{\xi'}(w_1')=t^{-2(m_1-l_1-1)},
\varphi(v_1',w_1')=t^{1-m_1},\psi_{\xi'}(v_2')=t^{2-2l_2},
\psi_{\xi'}(w_2')=0, \varphi(v_2',w_2')=t^{1-m_2},
\varphi(v_1',v_2')=\varphi(v_1',w_2')
=\varphi(w_1',v_2')=\varphi(w_1',w_2')=0. $

The form modules ${^*W}_{l_1}^0(m_1)\oplus {^*W}_{l_2}^\delta(m_2)$
and ${^*W}_{l_1}^\delta(m_1)\oplus {^*W}_{l_2}^0(m_2)$ are
equivalent if and only if there exists an $A$-module isomorphism
$g:V\rightarrow V$ such that
$\psi_{\xi'}(gv)=\psi_\xi(v),\varphi(gv,gw)=\varphi(v,w)$ for any
$v,w\in V$. Assume
\begin{eqnarray*}
&&gv_j=\sum\limits_{i=0}^{m_1-1}(a_{j,i}t^iv_1'+b_{j,i}t^iw_1')+\sum\limits_{i=0}^{m_2-1}(c_{j,i}t^{i}v_{2}'
+d_{j,i}t^{i}w_{2}'),\\
&&gw_j=\sum\limits_{i=0}^{m_1-1}(e_{j,i}t^iv_1'+f_{j,i}t^iw_1')+\sum\limits_{i=0}^{m_2-1}(g_{j,i}t^{i}v_{2}'
+h_{j,i}t^{i}w_{2}'),j=1,2.
\end{eqnarray*}
Then ${^*W}_{l_1}^0(m_1)\oplus {^*W}_{l_2}^\delta(m_2)$ and
${^*W}_{l_1}^\delta(m_1)\oplus {^*W}_{l_2}^0(m_2)$ are equivalent if
and only if the equations $\psi_{\xi'}(gv_i)=\psi_\xi(v_i),
\psi_{\xi'}(gw_i)=\psi_\xi(w_i),
\varphi(gv_i,gv_j)=\varphi(v_i,v_j),
\varphi(gv_i,gw_j)=\varphi(v_i,w_j),$\linebreak
$\varphi(gw_i,gw_j)=\varphi(w_i,w_j),i,j=1,2,$ have solutions.

If $l_1+l_2<m_1$, some equations are
$e_{1,2l_1-m_1}^2+e_{1,2l_1-m_1}=\delta$ (if $l_1\neq\frac{m_1}{2})$
or $ e_{1,0}^2+e_{1,0}f_{1,0}+\delta
f_{1,0}^2=0,a_{1,0}f_{1,0}+b_{1,0}e_{1,0}=1$ (if
$l_1=\frac{m_1}{2})$. As in the proof of Proposition
\ref{prop-nind}, we get "Artin-Schreier" equations which have no
solutions for $e_{1,2l_1-m_1}$ or $e_{1,0},f_{1,0}$ in $\tF_q$.
Hence ${^*W}_{l_1}^0(m_1)\oplus {^*W}_{l_2}^\delta(m_2)$ and
${^*W}_{l_1}^\delta(m_1)\oplus {^*W}_{l_2}^0(m_2)$ are not
equivalent.

If $l_1+l_2\geq m_1$, let
$gv_1=v_1',gw_1=w_1'+\sqrt{\delta}t^{l_1+l_2-m_1}v_2',gv_2=v_2',gw_2=w_2'+\sqrt{\delta}t^{l_1+l_2-m_2}v_1'$,
then this is a solution for the equations. It follows that
${^*W}_{l_1}^0(m_1)\oplus {^*W}_{l_2}^\delta(m_2)\cong
{^*W}_{l_1}^\delta(m_1)\oplus {^*W}_{l_2}^0(m_2)$.

\end{proof}
\begin{proposition}\label{prop-1}
The equivalence class of the module
\beq{^*W}_{l_1}(m_1)\oplus\cdots\oplus {^*W}_{l_s}(m_s),
[\frac{m_i}{2}]\leq l_i\leq m_i, l_i\geq l_{i+1},m_i-l_i\geq
m_{i+1}-l_{i+1}, i=1,\ldots,s,\eeq over $\bar{\tF}_q$ decomposes
into at most $2^k$
 equivalence classes over $\tF_q$, where \beq k=\#\{1\leq i\leq s|l_i+l_{i+1}<m_i \text{ and
 }l_i>\frac{m_i-1}{2}\}.\eeq
\end{proposition}
\begin{proof}
By Proposition \ref{prop-nind} and Remark \ref{rmk-1}, it is enough
to show that form modules of the form 
${^*W}_{l_1}^{\epsilon_1'}(m_1)\oplus\cdots\oplus
{^*W}_{l_s}^{\epsilon_s'}(m_{s})$, where $\epsilon_i'=0$ or
$\delta$, have at most $2^{k}$ equivalence classes. Suppose
$i_1,i_2,\ldots,i_{k}$ are such that $1\leq i_j\leq
s,l_{i_j}+l_{i_j+1}< m_{i_j},\
l_{i_j}>\frac{m_{i_j}-1}{2},j=1,\ldots,k$. Using Lemma \ref{lem-2-1}
one can easily show that a module of the above form is isomorphic to
one of the following modules: $V_1^{\epsilon_1}\oplus\cdots\oplus
V_{k}^{\epsilon_{k}}$, where
$V_{t}^{\epsilon_t}={^*W}_{l_{i_{t-1}+1}}^0(m_{i_{t-1}+1})\oplus\cdots\oplus
{^*W}_{l_{i_{t}-1}}^0(m_{i_{t}-1})\oplus
{^*W}_{m_{i_{t}}}^{\epsilon_t}(m_{i_{t}})$, $t=1,\ldots,k-1$,
$i_0=0$, and
$V_{k}={^*W}_{l_{i_{k-1}+1}}^0(m_{i_{k-1}+1})\oplus\cdots\oplus
{^*W}_{l_{i_{k}}}^{\epsilon_k}(m_{i_{k}})\oplus
{^*W}_{l_{i_{k}+1}}^{0}(m_{i_{k}+1})\oplus\cdots\oplus
{^*W}_{l_{s}}^0(m_{s})$, $\epsilon_t=0$ or $\delta$, $t=1,\ldots,k$.
Thus the proposition is proved.
\end{proof}

\begin{corollary}\label{cor-1}
The nilpotent orbit $(m_1)^2_{l_1}\cdots(m_s)^2_{l_s}$ in
$\Lg(\bar{{\tF}}_q)^*$ splits into at most $2^k$ $G(\tF_q)$-orbits
in $\Lg(\tF_q)^*$.
\end{corollary}

\begin{proposition}\label{prop-symp}
The number of nilpotent $G(\tF_q)$-orbits in $\Lg(\tF_q)^*$ is at
most $p_2(n)$.
\end{proposition}
\begin{proof}
Recall that we have  mapped the nilpotent orbits in
$\Lg(\bar{\tF}_q)^*$ bijectively to the set
$\{(\mu,\nu)||\mu|+|\nu|=n,\nu_i\leq \mu_i+1\}:=\Delta$. By
Corollary \ref{cor-1}, a nilpotent orbit in $\Lg(\bar{\tF}_q)^*$
corresponding to $(\mu,\nu)\in\Delta,\mu=(\mu_1,\mu_2,\ldots,\mu_s),
\nu=(\nu_1,\nu_2,\ldots,\nu_s)$ splits into at most $2^{k}$ orbits
in $\Lg(\tF_q)^*$, where $k=\#\{1\leq i\leq
s|\mu_{i+1}+1\leq\nu_i<\mu_i+1\}$. We associate to the orbit $2^{k}$
pairs of partitions as follows. Suppose $r_1,r_2,...,r_k$ are such
that $\mu_{r_{i}+1}+1\leq\nu_{r_i}<\mu_{r_i}+1,i=1,...,k$ and let
\begin{eqnarray*}
&&\mu^{1,i}=(\mu_{r_{i-1}+1},\ldots,\mu_{r_i}),
\nu^{1,i}=(\nu_{r_{i-1}+1},\ldots,\nu_{r_i}),\\
&&\mu^{2,i}=(\nu_{r_{i-1}+1}-1,\ldots,\nu_{r_i}-1),
\nu^{2,i}=(\mu_{r_{i-1}+1}+1,\ldots,\mu_{r_i}+1),i=1,\ldots,k,\\
&&\mu^{k+1}=(\mu_{r_{k}+1},\ldots,\mu_{s}),
\nu^{k+1}=(\nu_{r_{k}+1},\ldots,\nu_{s}). \end{eqnarray*} We
associate to $(\mu,\nu)$ the pairs of partitions
$(\tilde{\mu}^{\epsilon_1,\ldots,\epsilon_k},
\tilde{\nu}^{\epsilon_1,\ldots,\epsilon_k})$,
$$\tilde{\mu}^{\epsilon_1,\ldots,\epsilon_k}=
(\mu^{\epsilon_1,1},\mu^{\epsilon_2,2},\ldots,\mu^{\epsilon_k,k},\mu^{k+1}),
\tilde{\nu}^{\epsilon_1,\ldots,\epsilon_k}=
(\nu^{\epsilon_1,1},\nu^{\epsilon_2,2},\ldots,\nu^{\epsilon_k,k},\nu^{k+1}),$$
where $\epsilon_i\in\{1,2\},i=1,\ldots,k$. Notice that the pairs of
partitions $(\tilde{\mu}^{\epsilon_1,\ldots,\epsilon_k},
\tilde{\nu}^{\epsilon_1,\ldots,\epsilon_k})$ are distinct and among
them only $(\mu,\nu)=(\tilde{\mu}^{1,\ldots,1},
\tilde{\nu}^{1,\ldots,1})$ is in $\Delta$. One can verify that the
set of all pairs of partitions constructed as above for all
$(\mu,\nu)\in\Delta$ is in bijection with the set
$\{(\mu,\nu)||\mu|+|\nu|=n\}$, which has cardinality $p_2(n)$. It
follows that the number of nilpotent orbits in $\Lg(\tF_q)^*$ is
less that $p_2(n)$.
\end{proof}

\section{odd orthogonal groups}
In this section we study the nilpotent orbits in $\Lg^*$ where $G$
is an odd orthogonal group.
\subsection{}\label{ssec-1}
Let $V$ be a vector space of dimension $2n+1$ over $\tk$ equipped
with a non-degenerate quadratic form $\alpha:V\rightarrow \tk$. Let
$\beta:V\times V\rightarrow\tk$ be the bilinear form associated to
$\alpha$. The odd orthogonal group is defined as $G=O(2n+1)=\{g\in
GL(V)\ |\ \alpha(gv)=\alpha(v), \forall\ v \in V\}$ and its Lie
algebra is $\Lg=\Lo(2n+1)=\{x\in \mathfrak{gl}(V)\ |\ \beta(xv,v)=0,
\forall\ v\in V\text{ and tr}(x)=0\}$.  Let $\xi$ be an element of
$\Lg^*$. There exists $X\in \mathfrak{gl}(V)$ such that
$\xi(x)=\tr(X x)$ for any $x\in\Lg$. We define a bilinear form
$$\beta_\xi:V\times V\rightarrow
\tk, (v,w)\mapsto \beta(Xv,w)+\beta(v,Xw).$$

\begin{lemma}\label{lem-3-1}
 The bilinear form $\beta_\xi$ is well-defined.
\end{lemma}
\begin{proof}Recall that the space $\text{Alt}(V)$ of
alternate bilinear forms on $V$ coincides with the second exterior
power $\wedge^2(V^*)$ of $V^*$. Consider the following linear
mapping \beq \Phi:\End_\tk(V)\rightarrow
\wedge^2(V^*)=\text{Alt}(V),\quad X\mapsto\beta_X\eeq where
$\beta_X(v,w)=\beta(Xv,w)+\beta(v,Xw)$ for $v,w\in V$. It is easy to
see that $\Phi$ is $G=O(V)$-equivariant. One can show that
$\ker\Phi$ coincides with the orthogonal complement $\Lg^\p$ of
$\Lg=\mathfrak{o}(V)$ in $\End_\tk(V)$ under the nondegenerate trace
form. It follows that $\beta_\xi$ does not depend on the choice of
$X$.
\end{proof}

Assume $\xi\in\Lg^*$. We denote $(V_\xi,\alpha,\beta_\xi)$ the
vector space $V$ equipped with the quadratic form $\alpha$ and the
bilinear form $\beta_\xi$.

\begin{definition}
Assume $\xi,\zeta\in\Lg^*$. We say that $(V_\xi,\alpha,\beta_\xi)$
and $(V_\zeta,\alpha,\beta_\zeta)$ are equivalent if there exists a
vector space isomorphism $g:V_\xi\rightarrow V_\zeta$ such that
$\alpha(gv)=\alpha(v)$ and $\beta_\zeta(gv,gw)=\beta_\xi(v,w)$ for
any $v,w\in V_\xi$.
\end{definition}
\begin{lemma}
Two elements $\xi,\zeta\in\Lg^*$ lie in the same $G$-orbit if and
only if there exists $g\in G$ such that
$\beta_\zeta(gv,gw)=\beta_\xi(v,w) $ for any $ v,w\in V$.
\end{lemma}
\begin{proof}
Assume $\xi(x)=\tr(X x),\zeta(x)=\tr(X' x)$, $\forall x\in\Lg$.
Using similar argument as in the proof of Lemma \ref{lem-3-1}, one
can see that $\xi,\zeta$ lie in the same $G$-orbit if and only if
there exists $g\in G$ such that $\beta((gX g^{-1}+
X')v,w)+\beta(v,(gX g^{-1}+X')w)=0,\ \forall\ v,w\in V$.
\end{proof}

\begin{corollary}
Two elements $\xi,\zeta\in\Lg^*$ lie in the same $G$-orbit if and
only if $(V_\xi,\alpha,\beta_\xi)$ is equivalent to
$(V_\zeta,\alpha,\beta_\zeta)$.
\end{corollary}
\subsection{}\label{ssec-3-3}

From now on we assume that $\xi\in \Lg^*$ is nilpotent. Let
$(V_\xi,\alpha,\beta_\xi)$ be defined as in subsection \ref{ssec-1}.
Let $\lambda$ be a formal parameter. There exists a smallest integer
$m$ such that there exists a set of vectors $v_0,\ldots,v_m$ for
which
$\beta_\xi(\sum_{i=0}^{m}v_i\lambda^i,v)+\lambda\beta(\sum_{i=0}^{m}v_i\lambda^i,v)=0$
for any $v\in V$ (see Lemma \ref{lem-n-3} below). Lemmas
\ref{lem-n-3}-\ref{lem-e1} in the following extend some results in
\cite{LS}. (Most parts of the proofs are included in \cite{LS}. We
add some conditions about the quadratic form $\alpha$.)
\begin{lemma}\label{lem-n-3}
The vectors $v_0,\ldots,v_m$ (up to multiple) and $m\geq 0$ are
uniquely determined by $\beta_\xi$ and $\beta$. Moreover,
$\beta(v_i,v_j)=\beta_\xi(v_i,v_j)=0$, $i,j=0,\ldots,m$,
$\alpha(v_i)=0$, $i=0,\ldots,m-1$ and we can assume $\alpha(v_m)=1.$
\end{lemma}

\begin{proof} Since $\xi$ is nilpotent, we can find a
cocharacter $\phi:\textbf{G}_m\rightarrow G$ for which
$\xi\in\Lg^*(\phi,>0)$. Moreover, we can find $X\in
\End_\tk(V)(\phi,>0)$ such that $\xi(x)=\tr(Xx)$ for all $x\in\Lg$.

Let $w_0$ be a non-zero vector such that $\beta(w_0,-)=0$. Then
$w_0$ is unique up to a multiple. We have that $w_0\in V(\phi,0)$.
If $\beta_\xi(w_0,v)=0$ for all $v\in V$, then $m=0$ and we are
done.

Now assume $\beta_\xi(w_0,-)$ does not vanish on $V$. Fix $v\in V$.
It is easy to show that if $\beta_\xi(v,w_0)=0$, then there is
$v'\in V$ for which $\beta(v',-)=\beta_\xi(v,-)$. Moreover, if
$\tilde{w}_{i-1}\in V(\phi,\geq i-1)$, then one can show that for
all $\tilde{w}_i$ such that
$\beta(\tilde{w}_i,-)=\beta_\xi(\tilde{w}_{i-1},-)$, we have
$\tilde{w}_{i}\in V(\phi,\geq i)+\tk w_0$.

We define inductively a set of vectors $w_i$, $i=0,\ldots,m$, such
that $\beta(w_0,-)=0$, $\beta(w_i,-)=\beta_\xi(w_{i-1},-)$,
$\beta_\xi(w_{m},-)=0$ and $m$ is minimal. We have defined $w_0$.
Assume $w_{i-1}\in V(\phi,\geq i-1)$ is found. Then
$\beta_\xi(w_{i-1},w_0)=\beta(w_{i/2},w_{i/2})=0$ if $i$ is even and
$\beta_\xi(w_{i-1},w_0)=\beta_\xi(w_{(i-1)/2},w_{(i-1)/2})=0$ if $i$
is odd. We define $w_i$ to be the unique vector such that
$\beta(w_i,-)=\beta_\xi(w_{i-1},-)$ and $w_{i}\in V(\phi,\geq i)$.
One readily sees that we find a unique (up to multiple) set of
vectors $w_i$, $i=0,\ldots,m$, such that $\beta(w_0,-)=0$,
$\beta(w_i,-)=\beta_\xi(w_{i-1},-)$, $\beta_\xi(w_{m},-)=0$ and $m$
is minimal.

Since all $w_i\in V(\phi,\geq 0)$, we see that $\beta(w_i,w_j)=0$.
Since for $i>0$, $w_i\in V(\phi,> 0)$,  we see that $\alpha(w_i)=0$
for $i>0$. Since $X\in \End_\tk(\phi,>0)$, it follows that
$\beta_\xi(w_i,w_j)=0$. We take $v_i=w_{m-i}$. Moreover, we can
assume $\alpha(v_m)=\alpha(w_0)=1$.
\end{proof}

\begin{lemma}\label{lem-vu}
Assume $m\geq 1$. There exist $u_0,u_1,\ldots,u_{m-1}$ such that
$\beta(v_i,u_j)=\beta_\xi(v_{i+1},u_j)=\delta_{i,j}$,
$\beta(u_i,u_j)=\beta_\xi(u_i,u_j)=0$, $i,j=0,\ldots,m-1$,
$\alpha(u_i)=0,i=0,\ldots,m-1$, and furthermore,
$\beta(u_i,v)=\beta_\xi(u_{i-1},v)$, $i=1,\ldots,m-1$, for all $v\in
V$.
\end{lemma}
\begin{proof} Choose $u_0$ such that $\beta(u_0,v_i)=0$,
$i=1,\ldots,m-1$, $\beta(u_0,v_0)=1$ and $\alpha(u_0)=0$ (such $u_0$
exists). We find inductively a set of vectors $u_i$, $1\leq i\leq
m-1$ such that $\beta(u_i,-)=\beta_\xi(u_{i-1},-)$ and
$\alpha(u_i)=0$. Assume $u_{i-1},$ $1\leq i\leq m-1$ is found. Since
$\beta_\xi(u_{i-1},v_m)=\beta(u_0,v_{m-i})=0$ (note that $m-i\geq
1$), there exist a unique $u_i$ such that
$\beta(u_i,-)=\beta_\xi(u_{i-1},-)$ and $\alpha(u_i)=0$. (The
existence is as in the proof of Lemma \ref{lem-n-3} and the
uniqueness is guaranteed by the condition $\alpha(u_i)=0$.)

Now it follows that if $i<j$,
$\beta(v_i,u_j)=\beta_\xi(v_{i-1},u_{j-2})=\beta_\xi(v_0,u_{j-i-1})=0$;
 if $i>j$,
 $\beta(v_i,u_j)=\beta(v_{i+1},v_{j+1})=\beta(v_m,u_{j-i+m})=0$;
if $i=j$,
$\beta(v_i,u_i)=\beta(v_{i-1},u_{i-1})=\beta(v_0,u_{0})=1$.
Moreover, $\beta(u_i,u_{i+2k})=\beta(u_{i+k},u_{i+k})=0$,
$\beta(u_i,u_{i+2k+1})=\beta_\xi(u_{i+k},u_{i+k})=0$. It follows
that $\beta(u_i,u_j)=0$. Similarly $\beta_\xi(u_i,u_j)=0$. The
$u_i$'s satisfy the conditions desired.
\end{proof}

\begin{lemma} The vectors $v_0,v_1,\ldots,v_m, u_0,u_1,\ldots,u_{m-1}$ are linearly independent.
\end{lemma}
\begin{proof}
Assume $\sum_{i=0}^ma_iv_i+\sum_{i=0}^{m-1}b_iu_i=0$. Then
$\beta(\sum_{i=0}^ma_iv_i+\sum_{i=0}^{m-1}b_iu_i,u_j)=a_j=0$,
$\beta(\sum_{i=0}^ma_iv_i+\sum_{i=0}^{m-1}b_iu_i,v_j)=b_j=0,\
j=0,\ldots,m-1$ and
$\beta_\xi(\sum_{i=0}^ma_iv_i+\sum_{i=0}^{m-1}b_iu_i,u_{m-1})=a_m=0$.
\end{proof}
Let $V_{2m+1}$ be the vector subspace of $V$ spanned by
$v_0,v_1,\ldots,v_m, u_0,u_1,\ldots,u_{m-1}$. If $m=0$, let $W$ be a
complementary subspace of $V_{2m+1}$ in V. If $m\geq 1$, let
$W=\{w\in V_\xi|\beta(w,v)=\beta_{\xi}(w,v)=0,\ \forall\ v\in
V_{2m+1}\}$.
\begin{lemma}\label{lem-8}
We have $V_\xi=V_{2m+1}\perp_{\beta,\beta_{\xi}}W$.
\end{lemma}
\begin{proof}
Assume $m=0$. Lemma follows since by the definition of $v_0$ we have
$\beta(v_0,v)=\beta_\xi(v_0,v)=0$ for any $v\in V$. Assume $m\geq
1$. A vector $w$ is in $W$ if and only if
$\beta(w,v_i)=\beta_\xi(w,v_i)=0$, $i=0,\ldots,m$ and
$\beta(w,u_i)=\beta_\xi(w,u_i)=0$, $i=0,\ldots,m-1$. By our choice
of $v_i$ and $u_i$'s, we have $\beta(v_m,w)=\beta_\xi(v_0,w)=0$,
$\beta(w,v_i)=\beta_\xi(w,v_{i+1})$ and
$\beta(w,u_i)=\beta_\xi(w,u_{i-1})$. Hence $w\in W$ if and only if
$\beta(w,u_i)=\beta(w,v_i)=0$, $i=0,\ldots,m-1$ and
$\beta_\xi(w,u_{m-1})=0$.
 Thus $\dim W\geq\dim V_\xi-(2m+1)$. Now we show $V_{2m+1}\cap W=\{0\}$. Let $w=\sum_{i=0}^ma_iv_i+
 \sum_{i=0}^{m-1}b_iu_i\in V_{2m+1}\cap W$. We have $\beta(w,u_j)=
 a_j=0$, $\beta(w,v_j)=b_j=0,\ j=0,\ldots,m-1$, and
 $\beta_\xi(w,u_{m-1})=a_m=0$. Hence together with the dimension
 condition we get the conclusion.
\end{proof}

Let $V_\xi=V_{2m+1}\oplus W$ be as in Lemma \ref{lem-8}. Then we get
a $2(n-m)$ dimensional vector space $W$, equipped with a quadratic
form $\alpha|_W$ and a bilinear form $\beta_\xi|_{W\times W}$. It is
easily seen that the quadratic form $\alpha|_W$ is non-defective on
$W$, namely, $\beta|_{W\times W}$ is non-degenerate. Define a linear
map $T_\xi: W\rightarrow W$ by $$\beta(T_\xi w,w')=\beta_\xi(w,w'),
w,w'\in W.$$

\begin{lemma}\label{lem-e1}
Assume $V_\xi=V_{2m_\xi+1,\xi}\oplus W_\xi$ is equivalent to
$V_\zeta=V_{2m_\zeta+1,\zeta}\oplus W_\zeta$, then $m_\xi=m_\zeta$
and $(W_\xi,\beta,\beta_\xi)$ is equivalent to
$(W_\zeta,\beta,\beta_\zeta)$.
\end{lemma}
\begin{proof}
Assume $V_{2m_\xi+1,\xi}=\text{span}\{v_i^1,u_i^1\}$ and
$V_{2m_\zeta+1,\zeta}=\text{span}\{v_i^2,u_i^2\}$, where
$v_i^1,v_i^2$ are as in Lemma \ref{lem-n-3} and $u_i^1,u_i^2$ are as
in Lemma \ref{lem-vu}. By assumption, there exists
$$g:V_{2m_\xi+1,\xi}\oplus W_\xi\rightarrow
V_{2m_\zeta+1,\zeta}\oplus W_\zeta$$ such that
$\beta(gv,gw)=\beta(v,w)$ and $\beta_\zeta(gv,gw)=\beta_\xi(v,w)$.
Since for all $v\in V$, $\beta_\zeta(\sum_{i=0}^{m_\zeta}
v_i^2\lambda^i,v)+\lambda\beta(\sum_{i=0}^{m_\zeta}
v_i^2\lambda^i,v)=0$, we get $\beta_\xi( \sum_{i=0}^{m_\zeta}
g^{-1}v_i^2\lambda^i,v)+\lambda\beta(\sum_{i=0}^{m_\zeta}
g^{-1}v_i^2\lambda^i,v)=0$. Hence by Lemma \ref{lem-n-3},
$m_\xi=m_\zeta$ and $g^{-1}v_i^2\in V_{2m_\xi+1,\xi}$.

For $w\in W_\xi$, suppose $gw=\sum a_iv_i^2+\sum b_iu_i^2+w'$ where
$w'\in W_\zeta$. Since $g^{-1}v_i^2\in V_{2m_\xi+1,\xi}$, we have
$\beta_\xi(g^{-1}v_i^2,w)=0$, $i=0,\ldots,m_\xi$. It follows that
$\beta_\zeta(v_i^2,gw)=b_i=0$, $i=0,\ldots,m_\xi-1$. We get $gw=\sum
a_iv_i^2+w'$. Define $$\varphi:W_\xi\rightarrow W_\zeta,\  w\mapsto
gw \text{ projects to } W_\zeta.$$ Let $w_1,w_2\in W_{\xi}$. Assume
$gw_1=\sum a_i^1v_i^2+w_1'$, $gw_2=\sum a_i^2v_i^2+w_2'$. We have
$\beta(gw_1,gw_2)=\beta(w_1',w_2')=\beta(w_1,w_2),
\beta_\zeta(gw_1,gw_2)=\beta_\zeta(w_1',w_2')=\beta_\xi(w_1,w_2)$,
namely, $\beta(\varphi(w_1),\varphi(w_2))=\beta(w_1,w_2)$,
$\beta_\zeta(\varphi(w_1),\varphi(w_2))=\beta_\xi(w_1,w_2)$. Now we
show that $\varphi$ is a bijection. Let $w\in W_\xi$ be such that
$\varphi(w)=0$. Then for any $v\in W_\xi$,
$\beta(v,w)=\beta(\varphi(v),\varphi(w))=0$. Since
$\beta|_{W_\xi\times W_\xi}$ is nondegenerate, $w=0$. Thus $\varphi$
is injective. On the other hand, we have $\dim W_\xi=\dim W_\zeta$.
Hence $\varphi$ is bijective.
\end{proof}

\begin{corollary}
Assume $V_\xi=V_{2m_\xi+1,\xi}\oplus W_\xi$ is equivalent to
$V_\zeta=V_{2m_\zeta+1,\zeta}\oplus W_\zeta$, then $m_\xi=m_\zeta$
and $T_\xi$, $T_\zeta$ are conjugate.
\end{corollary}
\begin{lemma}\label{lem-5}
Assume $\xi$ is nilpotent. Then $T_\xi$ is nilpotent.
\end{lemma}
\begin{proof}We replace $G=Sp(V)$ by $G=O(V)$ and
$\beta$ by $\beta|_{W\times W}$ in the proof of Lemma
\ref{lem-nilp1}. Moreover, when apply $\Ad(\phi(a))$ to $T_\xi$, we
regard $\phi(a)$ as a linear map restricting to the subspace $W$ of
$V$ so that $\phi(a)\in O(W)$. Also notice that
$T_\xi\in\Lo(W)=\{x\in\mathfrak{gl}(W)|\beta(xw,w)=0,\ \forall\ w\in
W\}$, since $\beta(T_\xi w,w)=\beta_\xi(w,w)=0$ for all $w\in W$.
Then the same argument as in the proof of Lemma \ref{lem-nilp1}
applies since $\beta|_{W\times W}$ is nondegenerate.
\end{proof}

\subsection{} In this subsection assume $\tk$  is algebraically
closed. By Lemma \ref{lem-8}, every form module
$(V_\xi,\alpha,\beta_\xi)$ can be reduced to the form
$V_\xi=V_{2m+1}\oplus W_\xi$, where $V_{2m+1}$ has a basis
$\{v_i,i=0,\ldots,m,u_i,i=0,\ldots,m-1\}$ as in Lemmas \ref{lem-n-3}
and \ref{lem-vu}. We have that $(V_\xi,\alpha,\beta_\xi)$ is
determined by $V_{2m+1}$ and
$(W_\xi,\alpha|_{W_\xi},\beta_\xi|_{W_\xi\times W_\xi})$. Now we
consider $(W_\xi,\alpha|_{W_\xi},\beta_\xi|_{W_\xi\times
W_\xi}):=(W,\alpha|_{W},\beta_\xi|_{W\times W})$ and let $T_\xi:
W\rightarrow W$ be defined as in subsection \ref{ssec-3-3}. It
follows that $\beta_\xi|_{W\times W}$ is determined by $T_\xi$ and
$\beta|_{W\times W}$.

Since $T_\xi\in\Lo(W)$ is nilpotent (Lemma \ref{lem-5}), we can view
$W$ as a $k[T_\xi]-$module. By the classification of nilpotent
orbits in $\Lo(W)$ (see \cite{Hes}, sections 3.5 and 3.9), $W$ is
equivalent to $W_{l_1}(m_1)\oplus\cdots\oplus W_{l_s}(m_s)$ for some
$m_1\geq\cdots\geq m_s$, $l_1\geq\cdots\geq l_s$ and
$m_1-l_1\geq\cdots\geq m_s-l_s$, where $[(m_i+1)/2]\leq l_i\leq m_i$
(notation as in \cite{X}, Proposition 2.3).

\begin{lemma}
Assume $m<k-l$. We have $V_{2m+1}\oplus W_l(k)\cong V_{2m+1}\oplus
W_{k-m}(k)$.
\end{lemma}
\begin{proof}
Assume $V_{2m+1}=\text{span}\{v_0,\ldots,v_m, u_0,\ldots,u_{m-1}\}$,
where $v_i,u_i$ are chosen as in Lemma \ref{lem-n-3} and Lemma
\ref{lem-vu}. Assume $V_{2m+1}\oplus W_l(k)$ and $V_{2m+1}\oplus
W_{k-m}(k)$ correspond to $\xi_1$ and $\xi_2$ respectively. Let
$T_1=T_{\xi_1}:W_{l_1}(k)\rightarrow W_{l_1}(k)$ and
$T_2=T_{\xi_2}:W_{l_2}(k)\rightarrow W_{l_2}(k)$. There exist
$\rho_1,\rho_2$ such that
$W_l(k)=\text{span}\{\rho_1,\ldots,T_1^{k-1}\rho_1,\rho_2,\ldots,T_1^{k-1}\rho_2\}$,
$T_1^k\rho_1=T_1^k\rho_2=0$, $\alpha(T_1^i\rho_1)=\delta_{i,l-1},\
\alpha(T_1^i\rho_2)=0$,
$\beta(T_1^i\rho_1,T_1^j\rho_1)=\beta(T_1^i\rho_2,T_1^j\rho_2)=0$
and $\beta(T_1^i\rho_1,T_1^j\rho_2)=\delta_{i+j,k-1}$.  There exist
$\tau_1$, $\tau_2$ such that
$W_{k-m}(k)=\text{span}\{\tau_1,\ldots,T_2^{k-1}\tau_1,\tau_2,\ldots,T_2^{k-1}\tau_2\}$,
 $T_2^k\tau_1=T_2^k\tau_2=0$,
$\alpha(T_2^i\tau_1)=\delta_{i,k-m-1},\ \alpha(T_2^i\tau_2)=0$,
$\beta(T_2^i\tau_1,T_2^j\tau_1)=\beta(T_2^i\tau_2,T_2^j\tau_2)=0$
and $\beta(T_2^i\tau_1,T_2^j\tau_2)=\delta_{i+j,k-1}$. Define $g:
V_{2m+1}\oplus W_l(k)\rightarrow V_{2m+1}\oplus W_{k-m}(k)$ by
$gv_i=v_i,\ gu_i=u_i+(T_2^{k-(m+l)+i}+T_2^i)\tau_2,
gT_1^j\rho_2=T_2^j\tau_2,\
gT_1^j\rho_1=T_2^j\tau_1+v_{k-1-j}+v_{m+l-1-j} $, where $v_i=0,$ if
$i<0$ or $i>m$. Then $g$ is the isomorphism we want.
\end{proof}

\begin{lemma}\label{lem-9}
Assume $m\geq k-l_i,i=1,2$. We have $V_{2m+1}\oplus W_{l_1}(k)\cong
V_{2m+1}\oplus W_{l_2}(k)$ if and only if $ l_1=l_2$.
\end{lemma}
\begin{proof}
Assume $k-m\leq l_1< l_2$. We show that  $V_{2m+1}\oplus
W_{l_1}(k)\ncong V_{2m+1}\oplus W_{l_2}(k)$. Let
$(V_1,\alpha,\beta_1)=V_{2m+1}\oplus W_{l_1}(k)$ and
$(V_2,\alpha,\beta_2)=V_{2m+1}\oplus W_{l_2}(k)$. Let
$T_1=T_{\xi_1}:W_{l_1}(k)\rightarrow W_{l_1}(k)$ and
$T_2=T_{\xi_2}:W_{l_2}(k)\rightarrow W_{l_2}(k)$. Assume there
exists $g:V_{2m+1}\oplus W_{l_1}(k)\rightarrow V_{2m+1}\oplus
W_{l_2}(k)$ a linear isomorphism satisfying
$\beta_2(gv,gw)=\beta_1(v,w)$ and $\alpha(gv)=\alpha(v)$. Define
$\varphi:W_{l_1}(k)\rightarrow W_{l_2}(k)$ by $w_1\mapsto
(gw_1\text{ projects to }W_{l_2}(k))$. Then we have
$\beta(\varphi(w_1),\varphi(w_1'))=\beta(w_1,w_1')$,
$\beta_2(\varphi(w_1),\varphi(w_1'))=\beta_1(w_1,w_1')$ and
$T_2(\varphi(w))=\varphi(T_1(w))$ (see the proof of Lemma
\ref{lem-e1}).

Let $v_i,\ i=0,\ldots,m$, and $u_i,\ i=0,\ldots,m-1$, be a basis of
$V_{2m+1}$ as in Lemmas \ref{lem-n-3} and \ref{lem-vu}. Choose a
basis $T_i^{j}\rho_i,T_i^j\tau_i$, $j=0,\ldots,k-1$, $i=1,2$ of
$W_{l_i}(k)$ such that $T_i^{k}\rho_i=T_i^{k}\tau_i=0$,
$\beta(T_i^{j_1}\rho_i,T_j^{j_2}\tau_j)=\delta_{j_1+j_2,k-1}\delta_{i,j}$,
$\beta(T_i^{j_1}\rho_i,T_j^{j_2}\rho_j)=\beta(T_i^{j_1}\tau_i,T_j^{j_2}\tau_j)=0$,
$\alpha(T_i^j\rho_i)=\delta_{j,l_i-1}$ and $\alpha(T_i^j\tau_i)=0$.
We have
$$gv_i=av_i, i=0,\ldots,m,\
gu_i=u_i/a+\sum_{l=0}^{m}a_{il}v_l+\sum_{l=0}^{k-1}x_{il}T_2^l\rho_2+\sum_{l=0}^{k-1}y_{il}T_2^l\tau_2.$$
Now we can assume $$gT_1^j\rho_1=\sum_{i=0}^{k-1-j}
a_iT_2^{i+j}\rho_2+\sum_{i=0}^{k-1-j}
b_iT_2^{i+j}\tau_2+\sum_{i=0}^{m} c_{ij}v_i+\sum_{i=0}^{m-1}
d_{ij}u_i,\ j=0,\ldots,k-1,$$$$ gT_1^j\tau_1=\sum_{i=0}^{k-1-j}
e_iT_2^{i+j}\rho_2+\sum_{i=0}^{k-1-j}
f_iT_2^{i+j}\tau_2+\sum_{i=0}^{m} g_{ij}v_i+\sum_{i=0}^{m-1}
h_{ij}u_i,\ j=0,\ldots,k-1.$$ We have
\begin{eqnarray*}
&&\beta(gv_i,gT_1^j\rho_1)=\beta(v_i,T_1^j\rho_1)=0\Rightarrow
d_{ij}=0,\
i=0,\ldots,m-1,\ j=0,\ldots,k-1,\\
&&\beta(gv_i,gT_1^j\tau_1)=\beta(v_i,T_1^j\tau_1)=0\Rightarrow
h_{ij}=0,\ i=0,\ldots,m-1,\ j=0,\ldots,k-1, \end{eqnarray*}
\begin{eqnarray*}
&&\beta(gu_i,gT_1^j\rho_1)=\beta(u_i,T_1^j\rho_1)=0\Rightarrow
\frac{c_{ij}}{a}+\sum_{l=0}^{k-1-j} (x_{il}b_{k-1-j-l}+ y_{il}a_{k-1-j-l})=0,\\
&&\beta_\xi(gu_i,gT_1^j\rho_1)=\beta_\xi(u_i,T_1^j\rho_1)=0\Rightarrow
\frac{c_{i+1,j}}{a}+\sum_{l=0}^{k-2-j} (x_{il}b_{k-2-j-l}+
y_{il}a_{k-2-j-l})=0.
\end{eqnarray*}
The last two equations imply that
\begin{equation}\label{e-1}
c_{m,k-1}=0, c_{ij}=c_{i+1,j-1}, i=0,\ldots,m-1,j=0,\ldots,k-1.
\end{equation}
Similarly we have
\begin{equation}\label{e-2}
g_{m,k-1}=0, g_{ij}=g_{i+1,j-1}, i=0,\ldots,m-1,j=0,\ldots,k-1.
\end{equation}
We also have
\begin{equation*}
\alpha(gT_1^{l_1-1}\rho_1)=\alpha(T_1^{l_1-1}\rho_1)=1\Rightarrow
c_{m,l_1-1}^2+a_{l_2-l_1}^2+\sum_{i=0}^{k+1-2l_1}a_ib_{k+1-2l_1-i}=1,
\end{equation*}
\begin{equation*}
\alpha(gT_1^{l_2-1}\rho_1)=\alpha(T_1^{l_2-1}\rho_1)=0\Rightarrow
c_{m,l_2-1}^2+a_{0}^2+\sum_{i=0}^{k+1-2l_2}a_ib_{k+1-2l_2-i}=0,
\end{equation*}
\begin{eqnarray*}
\alpha(gT_1^{l_2-1}\tau_1)=\alpha(T_1^{l_2-1}\tau_1)=0\Rightarrow
g_{m,l_2-1}^2+e_{0}^2+\sum_{i=0}^{k+1-2l_2}a_ib_{k+1-2l_2-i}=0.
\end{eqnarray*}
Since $l_2>l_1\geq[k+1]/2$, we have $k+1-2l_2<0$. Thus we get
$a_0=c_{m,l_2-1}$ and $e_0=g_{m,l_2-1}$. Since $l_2>k-m$, by
equations (\ref{e-1}) and (\ref{e-2}) we get
$c_{m,l_2-1}=g_{m,l_2-1}=0$ (if $l_2=k$) or
$c_{m,l_2-1}=c_{0,l_2+m-1}=0,\ g_{m,l_2-1}=g_{0,l_2+m-1}=0$ (if
$l_2<k$). Thus $a_0=e_0=0$. But from
$\beta(g\rho_1,gT_1^{k-1}\tau_1)=\beta(\rho_1,T_1^{k-1}\tau_1)=1$ we
have $a_0f_0+e_0b_0=1$. This is a contradiction.
\end{proof}

It follows that for any $V=(V_{\xi},\alpha,\beta_{\xi})$, there
exist a unique $m\geq 0$ and a unique sequence of modules
$W_{l_i}(k_i)$, $i=1,\ldots,s$ such that $$V\cong V_{2m+1}\oplus
W_{l_1}(k_1)\oplus\cdots\oplus W_{l_s}(k_s),$$ $[(k_i+1)/2]\leq
l_i\leq k_i$, $k_1\geq k_2\geq\cdots\geq k_s$, $l_1\geq
l_2\geq\cdots\geq l_s$ and $m\geq k_1-l_1\geq k_2-l_2\geq\cdots\geq
k_s-l_s$. We call this the {\em normal form} of the module $V$. Two
form modules are equivalent if and only if their normal forms are
the same. Hence to each nilpotent orbits we associate a pair of
partitions
$$(m,k_1-l_1,\ldots,k_s-l_s)(l_1,\ldots,l_s)$$ where $l_1\geq
l_2\geq\cdots\geq l_s\geq 0$ and $m\geq k_1-l_1\geq
k_2-l_2\geq\cdots\geq k_s-l_s\geq 0$. This defines a bijection from
the set of nilpotent orbits to the set
$\{(\nu,\mu)|\nu=(\nu_0,\nu_1,\ldots,\nu_s),
\mu=(\mu_1,\mu_2,\ldots,\mu_s),|\mu|+|\nu|=n,\nu_i\leq\mu_i,i=1,\ldots,s\}$,
which has cardinality $p_2(n)-p_2(n-2)$.

\subsection{}
In this subsection, we classify the form modules
$(V_\xi,\alpha,\beta_\xi)$ over $\tF_q$. We have
$V_\xi=V_{2m+1}\oplus W_\xi$ for some $m$ and $W_\xi$ (Lemmas
\ref{lem-n-3}-\ref{lem-8} are valid over $\tF_q$). By the
classification of $(W_\xi,\alpha|_{W_\xi},T_\xi)$ over $\tF_q$, we
have $W_\xi\cong \oplus W_{l_i}^{\epsilon_i}(k_i)$ where
$\epsilon_i=0$ or $\delta$, $m_1\geq\cdots\geq m_s$,
$l_1\geq\cdots\geq l_s$, $m_1-l_1\geq\cdots\geq m_s-l_s$ and
$[(m_i+1)/2]\leq l_i\leq m_i$ (notation as in \cite{X}, Proposition
3.1).
\begin{lemma}\label{lem-c1}
Assume $m\geq k-l$ and $l>m$. We have $V_{2m+1}\oplus W_l^0(k)\cong
V_{2m+1}\oplus W_l^\delta(k)$.
\end{lemma}
\begin{proof}
Let $W_l^0(k)=(W_1,\alpha,T_1)$ and
$W_l^\delta(k)=(W_2,\alpha,T_2)$. Take $\rho_1,\rho_2$ such that
$W_l^0(k)=\text{span}\{\rho_1,\ldots,T_1^{k-1}\rho_1,\rho_2,\ldots,T_1^{k-1}\rho_2\}$,
$T_1^k\rho_1=T_1^k\rho_2=0$, $\alpha(T_1^{i}\rho_1)=\delta_{i,l-1}$,
$\alpha(T_1^i\rho_2)=0$,
$\beta(T_1^i\rho_1,T_1^j\rho_1)=\beta(T_1^i\rho_2,T_1^j\rho_2)=0$
and $\beta(T_1^i\rho_1,T_1^j\rho_2)=\delta_{i+j,k-1}$.  Take
$\tau_1,\tau_2$ such that
$W_{l}^\delta(k)=\text{span}\{\tau_1,\ldots,T_2^{k-1}\tau_1,\tau_2,\ldots,T_2^{k-1}\tau_2\}$
, $T_2^k\tau_1=T_2^k\tau_2=0$, $\alpha(T_2^i\tau_1)=\delta_{i,l-1},\
\alpha(T_2^i\tau_2)=\delta_{i,k-l}\delta$,
$\beta(T_2^i\tau_1,T_2^j\tau_1)=\beta(T_2^i\tau_2,T_2^j\tau_2)=0$
and $\beta(T_2^i\tau_1,T_2^j\tau_2)=\delta_{i+j,k-1}$. Let $v_i,u_i$
be a basis of $V_{2m+1}$ as in Lemmas \ref{lem-n-3} and
\ref{lem-vu}. Define $g:V_{2m+1}\oplus W_l^0(k)\rightarrow
V_{2m+1}\oplus W_l^\delta(k)$ by $gv_i=v_i,\
gu_i=u_i+\sqrt{\delta}T_2^{l-m-1+i}\tau_1,
gT_1^i\rho_1=T_2^i\tau_1,\
gT_1^i\rho_2=T_2^i\tau_2+\sqrt{\delta}v_{k-l+m-i}, $ where $v_i=0$
if $i<0$ or $i>m$.
\end{proof}
\begin{lemma}\label{lem-n1}
Assume $m\geq k-l$ and $l\leq m$. We have $V_{2m+1}\oplus
W_l^0(k)\ncong V_{2m+1}\oplus W_l^\delta(k)$.
\end{lemma}
\begin{proof}
Let $v_i,\ i=0,\ldots,m$ and $u_i,\ i=0,\ldots,m-1$ be a basis of
$V_{2m+1}$ as in Lemmas \ref{lem-n-3} and \ref{lem-vu}. Let
$W_l^0(k)=(W_0,\alpha,T_0)$ and
$W_l^\delta(k)=(W_\delta,\alpha,T_\delta)$. Choose a basis
$T_\epsilon^{j}\rho_\epsilon,T_\epsilon^j\tau_\epsilon$,
$j=0,\ldots,k-1$, $\epsilon=0,\delta$ of $W_{l}^{\epsilon}(k)$ such
that $T_\epsilon^{k}\rho_\epsilon=T_\epsilon^{k}\tau_\epsilon=0$,
$\beta(T_{\epsilon_1}^{j_1}\rho_{\epsilon_1},T_{\epsilon_2}^{j_2}\tau_{\epsilon_2})=\delta_{j_1+j_2,k-1}\delta_{\epsilon_1,\epsilon_2}$,
$\beta(T_{\epsilon_1}^{j_1}\rho_{\epsilon_1},T_{\epsilon_2}^{j_2}\rho_{\epsilon_2})
=\beta(T_{\epsilon_1}^{j_1}\tau_{\epsilon_1},T_{\epsilon_2}^{j_2}\tau_{\epsilon_2})=0$,
$\alpha(T_{\epsilon}^j\rho_{\epsilon})=\delta_{j,l-1}$ and
$\alpha(T_{\epsilon}^j\tau_{\epsilon})=\epsilon\delta_{j,k-l}\delta_{\epsilon,\delta}$.
Assume there exists $g:V_{2m+1}\oplus W_{l}^0(k)\rightarrow
V_{2m+1}\oplus W_{l_2}^\delta(k)$ a linear isomorphism satisfying
$\beta(gv,gw)=\beta(v,w)$, $\beta_\delta(gv,gw)=\beta_0(v,w)$ and
$\alpha(gv)=\alpha(v)$. We have
$$gv_i=av_i, i=0,\ldots,m,\
gu_i=u_i/a+\sum_{l=0}^{m}a_{il}v_l+\sum_{l=0}^{k-1}x_{il}T_\delta^l\rho_\delta+\sum_{l=0}^{k-1}y_{il}T_\delta^l\tau_\delta.$$
Now we can assume $$gT_0^j\rho_0=\sum_{i=0}^{k-1-j}
a_iT_\delta^{i+j}\rho_\delta+\sum_{i=0}^{k-1-j}
b_iT_\delta^{i+j}\tau_\delta+\sum_{i=0}^{m}
c_{ij}v_i+\sum_{i=0}^{m-1} d_{ij}u_i,\ j=0,\ldots,k-1,$$$$
gT_0^j\tau_0=\sum_{i=0}^{k-1-j}
e_iT_\delta^{i+j}\rho_\delta+\sum_{i=0}^{k-1-j}
f_iT_\delta^{i+j}\tau_\delta+\sum_{i=0}^{m}
g_{ij}v_i+\sum_{i=0}^{m-1} h_{ij}u_i,\ j=0,\ldots,k-1.$$ By similar
argument as in the proof of Lemma \ref{lem-9}, we get that $
c_{m,k-1}=0, g_{m,k-1}=0, c_{ij}=c_{i+1,j-1}, g_{ij}=g_{i+1,j-1},
i=0,\ldots,m-1,j=0,\ldots,k-1.$ Since we have $m\geq l$,
$c_{m,i}=c_{0,i+m}=0$ and $g_{m,i}=g_{0,i+m}=0$ when $i\geq k-l$. We
get some of the equations are $a_0^2+a_0b_0+\delta
b_0^2=1,e_0^2+e_0f_0+\delta f_0^2=0$ and $a_0f_0+b_0e_0=1$ (when
$l=(k+1)/2$) or
$a_{l-1-i}^2+\sum_{j=0}^{k-1-2i}a_jb_{k-1-2i-j}+\delta
b_{k-l-i}^2=\delta_{i,l-1}$,
$e_{l-1-i}^2+\sum_{j=0}^{k-1-2i}e_jf_{k-1-2i-j}+\delta
f_{k-l-i}^2=0,\ k-l\leq i\leq l-1$ and $a_0f_0+b_0e_0=1$ (when
$l>(k+1)/2$). We get $a_0=f_0=1,e_i=0,i=0,\ldots,2l-k-2$ and
$e_{2l-k-1}^2+e_{2l-k-1}+\delta=0$. This is a contradiction.
\end{proof}
Recall that we have the following lemma (\cite{X}, Lemma 4.4 (iii)).
\begin{lemma}\label{lem-6}
Assume $l_1\geq l_2$ and $k_1-l_1\geq k_2-l_2$. If $l_1+l_2>k_1$, we
have $W_{l_1}^0(k_1)\oplus W_{l_2}^0(k_2)\cong
W_{l_1}^\delta(k_1)\oplus W_{l_2}^\delta(k_2)$ and
$W_{l_1}^0(k_1)\oplus W_{l_2}^\delta(k_2)\cong
W_{l_1}^\delta(k_1)\oplus W_{l_2}^0(k_2)$.
\end{lemma}
Let $G(\tF_q)$, $\mathfrak{g}({\tF}_q)$ be the fixed points of a
Frobenius map $\mathfrak{F}_q$ relative to $\tF_q$ on $G$, $\Lg$.

\begin{proposition}\label{prop-2}
The nilpotent orbit in $\Lg^*$ corresponding to the pair of
partitions\linebreak
$(\nu_0,\nu_1,\ldots,\nu_s)(\mu_1,\mu_2,\ldots,\mu_s)$ splits into
at most $2^{k}$ $G(\tF_q)$-orbits in $\Lg(\tF_q)^*$, where
$k=\#\{i\geq 1|\nu_i<\mu_i\leq\nu_{i-1}\}$.
\end{proposition}
\begin{proof}
Let $V=V_{2m+1}\oplus W_{l_1}(\lambda_1)\oplus\cdots\oplus
W_{l_s}(\lambda_s)$ be the normal form of a module corresponding to
$(\nu,\mu)$ over $\bar{\tF}_q$. We show that the equivalence class
of $V$ over $\bar{\tF}_q$ decomposes into at most $2^k$ equivalence
classes over $\tF_q$. It is enough to show that form modules of the
form $V_{2m+1}\oplus
W_{l_1}^{\epsilon_1}(\lambda_1)\oplus\cdots\oplus
W_{l_s}^{\epsilon_s}(\lambda_s)$,  $\epsilon_i=0$ or $\delta$, have
at most $2^k$ equivalence classes over $\tF_q$. Suppose
$i_1,\ldots,i_k$ are such that
$\beta_{i_j}<\mu_{i_j}\leq\beta_{i_j-1}$, $j=1,\ldots,k$. Using
Lemma \ref{lem-c1}, \ref{lem-n1} and \ref{lem-6}, one can easily
verify that a form module of the above form is isomorphic to one of
the following modules: $V_1^{\epsilon_1}\oplus\cdots\oplus
V_k^{\epsilon_k}\oplus V_{k+1}$, where
$V_1^{\epsilon_1}=V_{2m+1}\oplus
W_{l_1}^0(\lambda_1)\oplus\cdots\oplus
W_{l_{i_1-1}}^0(\lambda_{i_1-1})\oplus
W_{l_{i_1}}^{\epsilon_1}(\lambda_{i_1})$,
$V_t^{\epsilon_t}=W_{l_{i_{t-1}+1}}^0(\lambda_{i_{t-1}+1})\oplus\cdots\oplus
W_{l_{i_{t}-1}}^0(\lambda_{i_t-1})\oplus
W_{l_{i_t}}^{\epsilon_t}(\lambda_{i_t})$, $t=2,\ldots,k$,
$\epsilon_i=0$ or $\delta$, $i=1,\ldots,k$, and
$V_{k+1}=W_{l_{i_k+1}}^0(\lambda_{i_k+1})\oplus\cdots\oplus
W_{l_s}^0(\lambda_s)$. Thus the proposition follows.
\end{proof}

\begin{proposition}\label{prop-orth}
The number of nilpotent $G(\tF_q)$-orbits in $\Lg(\tF_q)^*$ is at
most $p_2(n)$.
\end{proposition}
\begin{proof}
We have mapped the nilpotent orbits in $\Lg(\bar{\tF}_q)^*$
bijectively to the set $\{(\nu,\mu)|\nu=(\nu_0,\nu_1,\ldots,\nu_s),
\mu=(\mu_1,\mu_2,\ldots,\mu_s),|\mu|+|\nu|=n,\nu_i\leq\mu_i,i=1,\ldots,s\}:=\Delta$.
Let $(\nu,\mu)\in\Delta$,
$\nu=(\nu_0,\nu_1,\ldots,\nu_s),\mu=(\mu_1,\mu_2,\cdots,\mu_s)$. By
Proposition \ref{prop-2}, the nilpotent orbit corresponding to
$(\nu,\mu)$ splits into at most $2^k$ orbits in $\Lg(\tF_q)^*$,
where $k=\#\{i\geq 1|\nu_i<\mu_i\leq\nu_{i-1}\}$. We associate $2^k$
pairs of partitions to this orbit as follows.  Suppose
$r_1,r_2,...,r_k$ are such that
$\nu_{r_i}<\mu_{r_i}\leq\nu_{r_i-1},i=1,...,k$. Let
\begin{eqnarray*}
&&\nu^0=(\nu_0,\ldots,\nu_{r_1-1}),\mu^0=(\mu_1,\ldots,\mu_{r_1-1}),\\
&&\nu^{1,i}=(\nu_{r_{i}},\ldots,\nu_{r_{i+1}-1}),\mu^{1,i}=(\mu_{r_{i}},\ldots,\mu_{r_{i+1}-1}),\\&&
\nu^{2,i}=(\mu_{r_{i}},\ldots,\mu_{r_{i+1}-1}),\mu^{2,i}=(\nu_{r_{i}},\ldots,\nu_{r_{i+1}-1}),
i=1,\ldots,k-1,\\&&\nu^{1,k}=(\nu_{r_{k}},\ldots,\nu_{s}),\mu^{1,k}=(\mu_{r_{k}},\ldots,\mu_{s}),\\&&
\nu^{2,k}=(\mu_{r_{k}},\ldots,\mu_{s}),\mu^{2,k}=(\nu_{r_{k}},\ldots,\nu_{s}).
\end{eqnarray*}
We associate to $(\nu,\mu)$ the pairs of partitions
$(\tilde{\nu}^{\epsilon_1,\ldots,\epsilon_k},\tilde{\mu}^{\epsilon_1,\ldots,\epsilon_k}
)$, $$\tilde{\nu}^{\epsilon_1,\ldots,\epsilon_k}=
(\nu^0,\nu^{\epsilon_1,1},\nu^{\epsilon_2,2},\ldots,\nu^{\epsilon_k,k}),\tilde{\mu}^{\epsilon_1,\ldots,\epsilon_k}=
(\mu^0,\mu^{\epsilon_1,1},\mu^{\epsilon_2,2},\ldots,\mu^{\epsilon_k,k})
,$$ where $\epsilon_i\in\{1,2\}$, $i=1,\ldots,k$. Notice that the
pairs of partitions $(\tilde{\nu}^{\epsilon_1,\ldots,\epsilon_k},
\tilde{\mu}^{\epsilon_1,\ldots,\epsilon_k} )$ are distinct and among
them only
$(\nu,\mu)=(\tilde{\nu}^{1,\ldots,1},\tilde{\mu}^{1,\ldots,1} )$ is
in $\Delta$. One can verify that\linebreak
$\{(\tilde{\nu}^{\epsilon_1,\ldots,\epsilon_k},\tilde{\mu}^{\epsilon_1,\ldots,\epsilon_k}
)|(\nu,\mu)\in\Delta\}=\{(\nu,\mu)||\nu|+|\mu|=n\}$.
\end{proof}
\section{even orthogonal groups}
Let $V$ be a vector space of dimension $2n$ over $\tk$ equipped with
a non-defective quadratic form $\alpha:V\rightarrow \tk$. Let
$\beta:V\times V\rightarrow \tk$ be the non-degenerate bilinear form
associated to $\alpha$. The even orthogonal group is defined as
$G=O(2n)=\{g\in GL(V)\ |\ \alpha(gv)=\alpha(v), \forall\ v\in V\}$
and its Lie algebra is $\Lg=\Lo(2n)=\{x\in \mathfrak{gl}(V)\ |\
\beta(xv,v)=0, \forall\ v\in V \}$. Let $G(\tF_q)$,
$\mathfrak{g}({\tF}_q)$ be the fixed points of a split Frobenius map
$\mathfrak{F}_q$ relative to $\tF_q$ on $G$, $\Lg$.
\begin{proposition}\label{prop-3}
The numbers of nilpotent $G(\tF_q)$-orbits in $\Lg(\tF_q)$ and in
$\Lg(\tF_q)^*$ are the same.
\end{proposition}
The proposition can be proved in two ways.

First proof. We show that there exists a $G$-invariant
non-degenerate bilinear form on $\Lg=\Lo(2n)$ (G. Lusztig). Then we
can identify $\Lg$ and $\Lg^*$ via this bilinear form and the
proposition follows. Consider the vector space $\bigwedge^2V$ on
which $G$ acts in a natural way: $g(a\wedge b)=ga\wedge gb$. On
$\bigwedge^2V$ there is a $G$-invariant non-degenerate bilinear form
$$\langle a\wedge b,c\wedge d\rangle=\text{det}\left[
\begin{array}{cc}
\beta(a,c) & \beta(a,d) \\
\beta(b,c) & \beta(b,d)  \end{array} \right].$$ Define a
 map $\phi:\bigwedge^2 V\rightarrow \Lo(2n)$ by $a\wedge
b\mapsto \phi_{a\wedge b}$ and extending by linearity where $
\phi_{a\wedge b}(v)=\beta(a,v)b+\beta(b,v)a$. This map is
$G$-equivariant since we have $\phi_{ga\wedge gb}=g\phi_{a\wedge
b}g^{-1}$. One can easily verify that $\phi$ is a bijection. Define
$$\langle \phi_{a\wedge b},\phi_{c\wedge d}\rangle_{\Lo(2n)}=\langle
a\wedge b,c\wedge d\rangle$$ and extend it to $\Lo(2n)$ by
linearity. This defines a $G$-invariant non-degenerate bilinear form
on $\Lo(2n)$.

Second proof.   Let $\xi$ be an element of $\Lg^*$. There exists
$X\in \mathfrak{gl}(V)$ such that $\xi(x)=\tr(Xx)$ for any
$x\in\Lg$. We define a linear map $T_\xi: V\rightarrow V$ by
$$\beta(T_\xi v,v')=\beta(Xv,v')+\beta(v,Xv'),\text{ for all }v,v'\in V.$$

\begin{lemma}\label{lem-w3}
$T_\xi$ is well-defined.
\end{lemma}
\begin{proof}The same proof as in Lemma \ref{lem-3-1}
shows that $\beta(T_\xi v,v')$ is well-defined and thus $T_\xi$ is
well-defined.
\end{proof}
\begin{lemma}\label{lemma-1}
Two elements $\xi,\zeta\in\Lg^*$ lie in the same $G$-orbit if and
only if there exists $g\in G$ such that $gT_{\xi}g^{-1}=T_{\zeta}.$
\end{lemma}
\begin{proof}
Assume $\xi(x)=\tr(X_\xi x),\zeta(x)=\tr(X_\zeta x)$, $\forall
x\in\Lg$. Then $\xi,\zeta$ lie in the same $G$-orbit if and only if
there exists $g\in G$ such that $\tr(gX_\xi g^{-1}x)=\tr(X_\zeta
x),\ \forall\ x\in \Lg$. This is equivalent to $\beta((gX_\xi
g^{-1}+ X_\zeta)v,w)+\beta(v,(gX_\xi g^{-1}+X_\zeta)w)=0,\ \forall\
v,w \in V$, which is true if and only if $gT_{\xi}g^{-1}=T_{\zeta}.$
\end{proof}
Note that $\beta(T_\xi v,v)=0$ for any $v\in V$. Thus $T_\xi\in\Lg$.
We have in fact defined a bijection $\theta:\Lg^*\rightarrow \Lg$,
$\xi\mapsto T_\xi$. This induces a bijection
$\theta|_{\mathcal{N}'}:\mathcal{N}'\rightarrow \mathcal{N}$, where
$\mathcal{N}'$(resp. $\mathcal{N}$) is the set of all nilpotent
elements (unstable vectors) in $\Lg^*$(resp. $\Lg$). Moreover,
$\theta|_{\mathcal{N}'}$ is $G$-equivariant by Lemma \ref{lemma-1}.
The proposition follows.
\section{Springer correspondence}
In this section, we assume $\tk$ is algebraically closed. Throughout
subsections \ref{ssec-2}-\ref{ssec-3}, let $G$ be a simply connected
algebraic group of type $B,C$ or $D$ over $\tk$ and $\mathfrak{g}$
be the Lie algebra of $G$. Fix a Borel subgroup $B$ of $G$ with Levi
decomposition $B=TU$. We denote $r$ the dimension of $T$. Let $U^-$
be a maximal unipotent  subgroup opposite to $B$. Let
$\mathfrak{b},\mathfrak{t}$ , $\mathfrak{n}$ and $\Ln^-$ be the Lie
algebra of $B,T$,$U$ and $U^-$ respectively. Let $\cB$ be the
variety of Borel subgroups of $G$. Let $\Lg^*$ be the dual vector
space of $\Lg$. Let $$\Lt'=\{\xi\in\Lg^*|\xi(\Ln\oplus\Ln^-)=0\},
\Ln'=\{\xi\in\Lg^*|\xi(\Lb)=0\},\Lb'=\{\xi\in\Lg^*|\xi(\Ln)=0\}.$$
An element $\xi$ in $\Lg^*$ is called semisimple (resp. nilpotent)
if there exists $g\in G$ such that $g.\xi\in\Lt'$ (resp. $\Ln'$)(
see \cite{KW}). The proofs in this section are entirely similar to
those of \cite{Lu1,Lu4,X}. For completeness, we include the proofs
here.

\subsection{}\label{ssec-2}
Let $Z=\{(\xi,B_1,B_2)\in
\mathfrak{g}^*\times\cB\times\cB|\xi\in\Lb_1'\cap\Lb_2'\}$ and $
Z'=\{(\xi,B_1,B_2)\in
\mathfrak{g}^*\times\cB\times\cB|\xi\in\Ln_1'\cap\Ln_2'\}. $  Let
$\rc$ be a nilpotent orbit in $\Lg^*$.

\begin{lemma}\label{prop-dim}
$\mathrm{(i)}$ We have $\dim(\rc\cap\Ln')\leq\frac{1}{2}\dim \rc$.

$\mathrm{(ii)}$ Given $\xi\in \rc$, we have
$\dim\{B_1\in\cB|\xi\in\Ln_1'\}\leq(\dim G-r-\dim \rc)/2$.

$\mathrm{(iii)}$ We have $\dim Z=\dim G$ and $\dim Z'=\dim G-r$.
\end{lemma}
\begin{proof}
We have a partition
$Z=\cup_{\mathcal{\mathcal{O}}}Z_{\mathcal{\mathcal{O}}}$ according
to the $G$-orbits ${\mathcal{\mathcal{O}}}$ on $\cB\times\cB$ where
$Z_\mathcal{O}=\{(\xi,B_1,B_2)\in Z|(B_1,B_2)\in\mathcal{O}\}$.
Define in the same way a partition
$Z'=\cup_{\mathcal{\mathcal{O}}}Z'_{\mathcal{\mathcal{O}}}$.
Consider the maps from $Z_{\cO}$ and $Z'_{\cO}$ to $\cO$:
$(\xi,B_1,B_2)\mapsto (B_1,B_2)$. We have $\dim Z_{\cO}=\dim
(\Lb_1'\cap\Lb_2')+\dim\cO=\dim (\Lb_1\cap\Lb_2)+\dim\cO=\dim G$ and
$\dim Z_{\cO}'=\dim (\Ln_1'\cap\Ln_2')+\dim\cO=\dim
(\Ln_1\cap\Ln_2)+\dim\cO=\dim G-r$. Thus $\mathrm{(iii)}$ follows.

Let $Z'(\rc)=\{(\xi,B_1,B_2)\in Z'|\xi\in \rc\}\subset Z'$. From
(iii), we have $\dim Z'(\rc)\leq \dim G-r$. Consider the map
$Z'(\rc)\rightarrow \rc,\ (\xi,B_1,B_2)\mapsto \xi$. We have $\dim
Z'(\rc)=\dim\rc+2\dim\{B_1\in\cB|\xi\in\Ln_1'\}\leq \dim G-r$. Thus
(ii) follows.

Consider the variety $\{(\xi,B_1)\in \rc\times\cB|\xi\in\Ln_1'\}$.
By projecting it to the  first coordinate, and using (ii), we see
that it has dimension $\leq(\dim G-r+\dim \rc)/2$. If we project it
to the second coordinate, we get $\dim(\rc\cap\Ln')+\dim\cB\leq(\dim
G-r +\dim \rc)/2$ and (i) follows.
\end{proof}

\subsection{}
Recall that an element $\xi$ in $\Lg^*$ is called regular if the
connected centralizer $Z_G^0(\xi)$ in $G$ is a maximal torus of $G$
(\cite{KW}).
\begin{lemma}[\cite{KW}, Lemma 3.2]\label{ssreg}
There exist regular semisimple elements in $\Lg^*$ and they form an
open dense subset in $\Lg^*$.
\end{lemma}

\begin{remark}
Lemma \ref{ssreg} is not always true when $G$ is not simply
connected.
\end{remark}

\subsection{}
Let $\Lt_0',Y'$ be the set of semisimple regular elements in
$\Lt',\Lg^*$ respectively. By Lemma \ref{ssreg}, $\dim Y'=\dim G$.
Let $$\widetilde{Y}'=\{(\xi,gT)\in Y'\times G/T|g^{-1}.\xi\in
\Lt_0'\}.$$ Define $$\pi': \widetilde{Y}'\rightarrow Y'\text{ by  }
\pi'(\xi,gT)=\xi.$$ The Weyl group $W=NT/T$ acts (freely) on
$\widetilde{Y}'$ by $n:(\xi,gT)\mapsto(\xi,gn^{-1}T)$.
\begin{lemma}\label{lem-3}
$\pi': \widetilde{Y}'\rightarrow Y'$ is a principal $W$-bundle.
\end{lemma}
\begin{proof}
We show that if $\xi\in \Lg^*, g,g'\in G$ are such that
$g^{-1}.\xi\in\Lt_0'$ and $g'^{-1}.\xi\in\Lt_0'$, then $g'=gn^{-1}$
for some $n\in NT$. Let
$g^{-1}.\xi=\xi_1\in\Lt_0',g'^{-1}.\xi=\xi_2\in\Lt_0'$, then we have
$Z_G^0(\xi)=Z_G^0(g.\xi_1)=gZ_G^0(\xi_1)g^{-1}=gTg^{-1}$, similarly,
$Z_G^0(\xi)=Z_G^0(g'.\xi_2)=g'Z_G^0(\xi_2)g'^{-1}=g'Tg'^{-1}$, hence
$g'^{-1}g\in NT$.
\end{proof}

Let $$X'=\{(\xi,gB)\in\Lg^*\times G/B|g^{-1}.\xi\in\Lb'\}.$$ Define
$$\varphi':X'\rightarrow\Lg^* \text{ by }\varphi'(\xi,gB)=\xi.$$ The map
$\varphi'$ is $G$-equivariant with $G$-action on $X'$ given by
$g_0:(\xi,gB)\mapsto(g_0.\xi,g_0gB)$.
\begin{lemma}
$\mathrm{(i)}$ $X'$ is an irreducible variety of dimension equal to
$\dim G$.

$\mathrm{(ii)}$ $\varphi'$ is proper and $\varphi'(X')=\Lg^*$.

$\mathrm{(iii)}$ $(\xi,gT)\rightarrow(\xi,gB)$ is an isomorphism
$\rho:\widetilde{Y}'\xrightarrow{\sim}\varphi'^{-1}(Y')$.
\end{lemma}
\begin{proof}
(i) and (ii) are easy.

(iii) We first show that $\rho$ is a bijection. Suppose
$(\xi_1,g_1T),(\xi_2,g_2T)\in\widetilde{Y}'$ are such that
$(\xi_1,g_1B)=(\xi_2,g_2B)$, then we have
$g_1^{-1}.\xi_1\in\Lt_0',g_2^{-1}.\xi_2\in\Lt_0'$ and
$\xi_1=\xi_2,g_2^{-1}g_1\in B$. Similar argument as in the proof of
Lemma \ref{lem-3} shows $g_2^{-1}g_1\in NT$, hence $g_2^{-1}g_1\in
B\cap NT=T$ and it follows that $g_1T=g_2T$. Thus $\rho$ is
injective.  For $(\xi,gB)\in\varphi'^{-1}(Y')$, we have $\xi\in
Y',g^{-1}.\xi\in\Lb'$, hence there exists $b\in B,\xi_0\in\Lt_0'$
such that $g^{-1}.\xi=b.\xi_0$. Then $\rho(\xi,gbT)=(\xi,gB)$ and it
follows that $\rho$ is surjective.

Now we show that $\rho$ is an isomorphism of varieties. The proof is
entirely similar to the Lie algebra case (see for example
\cite{Jan}, Lemma 13.4). Let $\Lb'_0$ be the set of regular
semisimple elements in $\Lb'$. Consider the natural projection maps
$$f:\widetilde{Y}'\rightarrow G/T,\quad f':X'\rightarrow G/B.$$ Let
$U^-$ be as in the first paragraph of this section. Then $U^-B/B$
(resp. $U^-B/T$) is an open subset in $G/B$ (resp. $G/T$). We have
isomorphisms
\begin{eqnarray*}&&\Lt'_0\times U\times
U^-\xrightarrow{\sim}f^{-1}(U^-B/T),(\xi,u,u^-)\mapsto((u^-u).\xi,u^-uT)\\
&&\Lb'_0\times
U^-\xrightarrow{\sim}f'^{-1}(U^-B/B),(\xi,u^-)\mapsto(u^-.\xi,u^-B).
\end{eqnarray*}
Notice that $f^{-1}(U^-B/T)$ is the inverse image of
$f'^{-1}(U^-B/B)$ under $\rho$. Hence under the two isomorphisms
above, the map $\rho$ corresponds to the following isomorphism
$$\Lt'_0\times U\times
U^-\xrightarrow{\sim}\Lb'_0\times
U^-,\quad(\xi,u,u^-)\mapsto(u.\xi,u^-).$$ It follows that
$\rho^{-1}$ is a morphism on $f'^{-1}(U^-B/B)$. Since
$f'^{-1}(gU^-B/B)$ with $g\in G$ cover $\varphi'^{-1}(Y')$ and
$\rho,f'$ are $G$-equivariant, we see that $\rho^{-1}$ is a morphism
everywhere.
\end{proof}

By Lemma \ref{lem-3}, the map $\pi':\widetilde{Y}'\rightarrow Y'$ is
quasi-finite. Since $\pi'$ is proper, it follows that $\pi'$
 is a finite covering (see \cite{Mil}, I 1.10). Thus
$\pi'_!\bar{\mathbb{Q}}_{l\widetilde{Y}'}$ is a well-defined local
system on $Y'$ and the intersection cohomology complex
$IC(\Lg^*,\pi'_!\bar{\mathbb{Q}}_{l\widetilde{Y}'})$ is
well-defined.
\begin{proposition}\label{p-2}
$\varphi'_!\bar{\mathbb{Q}}_{lX'}$ is canonically isomorphic to
$IC(\Lg^*,\pi'_!\bar{\mathbb{Q}}_{l\widetilde{Y}'})$. Moreover,
$\End(\varphi'_!\bar{\mathbb{Q}}_{lX'})=\End(\pi'_!\bar{\mathbb{Q}}_{l\widetilde{Y}'})=\bar{\bQ}_l[W]$.
\end{proposition}
\begin{proof}
Using $\rho:\widetilde{Y}'\xrightarrow{\sim}\varphi'^{-1}(Y')$ and
$\bar{\mathbb{Q}}_{lX'}|_{\varphi'^{-1}(Y')}\cong\bar{\mathbb{Q}}_{l\widetilde{Y}'}$,
we have
$\varphi'_!\bar{\mathbb{Q}}_{lX'}|_{Y'}=\pi_!\bar{\mathbb{Q}}_{l\widetilde{Y}'}$
by base change theorem (see for example \cite{Mil}). Since
$\varphi'$ is proper and $X'$ is smooth of dimension equal to $\dim
Y'$, we have that the Verdier dual (see for example \cite{Jan},
12.13)
$\mathfrak{D}(\varphi'_!\bar{\mathbb{Q}}_{lX'})=\varphi'_!(\mathfrak{D}\bar{\mathbb{Q}}_{lX'})\cong\varphi'_!\bar{\bQ}_{lX'}[2\dim
Y']$. Hence by the definition of intersection cohomology complex, it
is enough to prove that
$$
\forall\
i>0,\dim\text{supp}\mathcal{H}^i(\varphi'_!\bar{\mathbb{Q}}_{lX'})<\dim
Y'-i.
$$
For $\xi\in\Lg^*$, the stalk
$\cH^i_\xi(\varphi'_!\bar{\mathbb{Q}}_{lX'})$ coincides with
$H^i_c(\varphi'^{-1}(\xi),\bar{\mathbb{Q}}_{l})$. Hence it is enough
to show $\forall\
i>0,\dim\{\xi\in\Lg^*|H^i_c(\varphi'^{-1}(\xi),\bar{\mathbb{Q}}_{l})\neq
0\}<\dim Y'-i$. If
$H^i_c(\varphi'^{-1}(\xi),\bar{\mathbb{Q}}_{l})\neq 0$, then $i\leq
2\dim \varphi'^{-1}(\xi)$. Hence it is enough to show that
$$\forall\ i>0,\dim\{\xi\in\Lg^*|\dim\varphi'^{-1}(\xi)\geq
i/2\}<\dim Y'-i.$$ Suppose this is not true for some $i$, then
$\dim\{\xi\in\Lg^*|\dim\varphi'^{-1}(\xi)\geq i/2\}\geq\dim Y'-i$.
Let $V=\{\xi\in\Lg^*|\dim\varphi'^{-1}(\xi)\geq i/2\}$, it is closed
in $\Lg^*$ but not equal to $\Lg^*$. Consider the map
$p:Z\rightarrow \Lg^*,\ (\xi,B_1,B_2)\mapsto \xi$. We have $\dim
p^{-1}(V)=\dim V+2\dim\varphi'^{-1}(\xi)\geq\dim V+i\geq\dim Y'$
(for some $\xi\in V$).
 Thus by Lemma \ref{prop-dim}
(iii), $p^{-1}(V)$ contains some $Z_\mathcal{O},\mathcal{O}=G$-orbit
of $(B,nBn^{-1})$ in $\cB\times\cB$, where $n$ is some element in
$NT$. If $\xi\in \Lt_0'$, then $(\xi,B,nBn^{-1})\in
Z_{\mathcal{O}}$, hence $\xi$ belongs to the projection of
$p^{-1}(V)$ to $\Lg^*$ which has dimension $\dim V<\dim Y'$. But
this projection is $G$-invariant hence contains all $Y'$. We get a
contradiction.

Since $\pi'$ is a principal $W$-bundle, we have
$\End(\pi'_!\bar{\mathbb{Q}}_{l\widetilde{Y}'})=\bar{\bQ}_l[W]$ (see
for example \cite{Jan}, Lemma 12.9). It follows that
$\End(\varphi'_!\bar{\mathbb{Q}}_{lX'})=\bar{\bQ}_l[W]$.
\end{proof}

\subsection{}
In this subsection, we introduce some sheaves on the variety of
semisimple $G$-orbits in $\Lg^*$ similar to \cite{Lu1,Lu4}.
\begin{lemma}[\cite{KW}, Theorem 4 (ii)]
A $G$-orbit in $\Lg^*$ is closed if and only if it consists of
semisimple elements.
\end{lemma}

Let $\A$ be the set of closed $G$-orbits in $\Lg^*$. By geometric
invariant theory, $\A$ has a natural structure of affine variety and
there is a well-defined morphism $\sigma:\Lg^*\rightarrow\A$ such
that $\sigma(\xi)$ is the $G$-conjugacy class of $\xi_s$, where
$\xi=\xi_s+\xi_n$ is the Jordan decomposition of $\xi$ (see
\cite{KW} for the notion of Jordan decomposition for $\xi\in\Lg^*$).
There is a unique $\varsigma\in\A$ such that
$\sigma^{-1}(\varsigma)=\{\xi\in\Lg^*|\xi \text{ nilpotent}\}$.

Recall that $Z=\{(\xi,B_1,B_2)\in
\mathfrak{g}^*\times\cB\times\cB|\xi\in\mathfrak{b}_1'\cap\mathfrak{b}_2'\}$.
Define $\tilde{\sigma}: Z\rightarrow\A$ by
$\tilde{\sigma}(\xi,B_1,B_2)=\sigma(\xi)$. For $a\in\A$, let
$Z^a=\tilde{\sigma}^{-1}(a)$.
\begin{lemma}
 We have $\dim Z^a\leq d_0$, where $d_0=\dim G-r$.
\end{lemma}
\begin{proof}
Define $m:Z^a\rightarrow\sigma^{-1}(a)$ by $(\xi,B_1,B_2)\mapsto
\xi$. Let $\rc\subset\sigma^{-1}(a)$ be a $G$-orbit. Consider
$m:m^{-1}(\rc)\rightarrow\rc$. We have $\dim m^{-1}(\rc)\leq\dim
\rc+2(\dim G-r-\dim\rc)/2=\dim G-r$ (use Lemma \ref{prop-dim} (ii)).
Since $\sigma^{-1}(a)$ is a union of finitely many $G$-orbits, it
follows that $\dim Z^a\leq d_0$.
\end{proof}

Let $\mathcal{T}=\cH^{2d_0}\tilde{\sigma}_!\bar{\bQ}_{lZ}$. Recall
that we set $ Z_\cO=\{(\xi,B_1,B_2)\in Z|(B_1,B_2)\in \cO\}$, where
$\cO$ is an orbit of $G$ action on $\cB\times \cB$. Let
$\mathcal{T}^\cO=\cH^{2d_0}\sigma^0_!\bar{\bQ}_l$, where $\sigma^0:\
 Z_\cO\rightarrow\A$ is the restriction of $\tilde{\sigma}$ on
$Z_\cO$.

\begin{lemma}\label{lc-1}
We have $\cT^\cO\cong\bar{\sigma}_!\bar{\bQ}_l$, where
$\bar{\sigma}:\Lt'\rightarrow\A$ is the restriction of $\sigma$.
\end{lemma}

\begin{proof}
The fiber of the natural projection $pr_{23}:Z_\cO\rightarrow \cO$
at $(B,nBn^{-1})\in \cO$ (where $n\in NT$) can be identified with
$V=\Lb'\cap n.\Lb'$. Let $\cT'^\cO=\cH^{2d_0-2\dim
\cO}\sigma'_!\bar{\bQ}_l$, where $\sigma':V\rightarrow\A$ is
$\xi\mapsto\sigma(\xi)$. Let $\cT''^{\cO}=\cH^{2d_0+2\dim
H}\sigma''_!\bar{\bQ}_l$, where $H=B\cap nBn^{-1}$ and
$\sigma'':G\times V\rightarrow\A$ is $(g,\xi)\mapsto\sigma(\xi)$.
Consider the composition $G\times
V\xrightarrow{pr_2}V\xrightarrow{\sigma'}\A$ (equal to $\sigma''$)
and the composition $G\times V\xrightarrow{p} H\backslash(G\times
V)=Z_\cO\xrightarrow{\sigma^0}\A$ (equal to $\sigma''$), we obtain
$$\cT''^{\cO}=\cH^{2d_0+2\dim
H}(\sigma'_!pr_{2!}\bar{\bQ}_l)=\cH^{2d_0+2\dim
H}(\sigma'_!\bar{\bQ}_l[-2\dim
G])=\cT'^{\cO},$$$$\cT''^{\cO}=\cH^{2d_0+2\dim
H}(\sigma^0_!p_{!}\bar{\bQ}_l)=\cH^{2d_0+2\dim
H}(\sigma^0_!\bar{\bQ}_l[-2\dim H])=\cT^{\cO}.$$ It follows that
$\cT^{\cO}=\cT'^{\cO}=\cH^{2d_0-2\dim \cO}(\sigma'_!\bar{\bQ}_l)$,
since $\dim\cO=\dim G-\dim H$. Now the map $\sigma':V\rightarrow\A$
factors as $V\xrightarrow{a'}\Lt'\xrightarrow{\bar{\sigma}} \A$.
Since the map $a'$ has fibers $\Ln'\cap n.\Ln'$ isomorphic to affine
spaces of dimension $d_0-\dim \cO$, we have
$a'_!\bar{\bQ}_l\cong\bar{\bQ}_l[-2(d_0-\dim \cO)]$ (see \cite{KiW},
VI Lemma 2.3). Hence $\cT^{\cO}\cong\cH^{2d_0-2\dim
\cO}(\bar{\sigma}_!a'_!\bar{\bQ}_l)\cong\cH^0\bar{\sigma}_!\bar{\bQ}_l$.
Since $\bar{\sigma}$ is a finite covering (Lemma \ref{l-fc}), it
follows that $\cT^{\cO}\cong\bar{\sigma}_!\bar{\bQ}_l$.
\end{proof}
\begin{lemma}\label{l-fc}
The map $\bar{\sigma}:\Lt'\rightarrow \A$ is a finite covering.
\end{lemma}
\begin{proof} By \cite{KW} Theorem 4 (i), the natural
map $\tk[\Lg^*]^G\rightarrow\tk[\Lt']^W$ is an isomorphism of
algebras. It follows that the variety $\A$  is isomorphic to
$\Lt'/W$. Thus the map $\bar{\sigma}$ is finite.
\end{proof}

Denote $\cT_\varsigma$ and $\cT^{\cO}_\varsigma$ the stalk  of $\cT$
and $\cT^{\cO}$ at $\varsigma$ respectively.
\begin{lemma}\label{p-1}
For $w\in W$, let $\cO_w$ be the $G$-orbit on $\cB\times\cB$ which
contains $(B,n_wBn_w^{-1})$. There is a canonical isomorphism
$\mathcal{T}_\varsigma\cong\bigoplus_{w\in
W}\mathcal{T}^{\cO_w}_\varsigma$.
\end{lemma}
\begin{proof}
We have
$\tilde{\sigma}^{-1}(\varsigma)=Z'=\{(\xi,B_1,B_2)\in\Lg^*\times\cB\times\cB|
\xi\in\Ln_1'\cap\Ln_2'\}$. We have a partition $Z'=\sqcup_{w\in
W}Z'_{\cO_w}$, where $Z'_{\cO_w}=\{(\xi,B_1,B_2)\in Z'|(B_1,B_2)\in
\cO_w\}$. Since $\dim Z'=d_0$, we have an isomorphism
\begin{equation*}
H^{2d_0}_c(Z',\bar{\bQ}_{l})= \bigoplus_{w\in
W}H^{2d_0}_c(Z'_{\cO_w},\bar{\bQ}_{l}),
\end{equation*}
which is $\mathcal{T}_\varsigma\cong\bigoplus_{w\in
W}\mathcal{T}_\varsigma^{\cO_w}$.
\end{proof}

Recall that we have
$\bar{\bQ}_l[W]=\End(\pi'_!\bar{\bQ}_{l\widetilde{Y}'})=\End(\varphi'_!\bar{\bQ}_{lX'})$.
In particular, $\varphi'_!\bar{\bQ}_{lX'}$ is naturally a $W$-module
and $\varphi'_!\bar{\bQ}_{lX'}\otimes\varphi'_!\bar{\bQ}_{lX'} $ is
naturally a $W$-module (with $W$ acting on the first factor). This
induces a $W$-module structure on
$\cH^{2d_0}\sigma_!(\varphi'_!\bar{\bQ}_{lX'}\otimes\varphi'_!\bar{\bQ}_{lX'})=\cT$.
Hence we obtain a $W$-module structure on the stalk $\cT_\varsigma$.

\begin{lemma}\label{l-multi}
Let $w\in W$. Multiplication by $w$ in the $W$-module structure of
$\cT_\varsigma=\bigoplus_{w'\in W}\cT_\varsigma^{\cO_{w'}}$ defines
for any $w'\in W$ an isomorphism
$\cT_\varsigma^{\cO_{w'}}\xrightarrow{\sim}\cT^{\cO_{ww'}}_\varsigma$.
\end{lemma}

\begin{proof}
We have an isomorphism
\begin{equation*}
f:Z'_{\cO_{w'}}\xrightarrow{\sim}Z'_{\cO_{ww'}},
(\xi,gBg^{-1},gn_{w'}Bn_{w'}^{-1}g^{-1})\mapsto(\xi,gn_w^{-1}Bn_wg^{-1},gn_{w'}Bn_{w'}^{-1}g^{-1}).
\end{equation*}
This induces an isomorphism
$$H^{2d_0}_c(Z'_{\cO_{w'}},\bar{\bQ}_{l})\xrightarrow{\sim}
H^{2d_0}_c(Z'_{\cO_{ww'}},\bar{\bQ}_{l})$$ which is just
multiplication by $w$.
\end{proof}

\subsection{}\label{ssec-3}
Let $\hat W$ be the set of simple modules (up to isomorphism) for
the Weyl group $W$ of $G$ (A description of $\hat W$ is given for
example in \cite{Lu3}). Given a semisimple object $M$ of some
abelian category such that $M$ is a $W$-module, we write
$M_\rho=\Hom_{\bar{\bQ}_l[W]}(\rho,M)$ for $\rho\in \Hat{W}$. We
have $M=\oplus_{\rho\in \Hat{W}}(\rho\otimes M_\rho)$ with $W$
acting on the $\rho$-factor and $M_\rho$ is in our abelian category.
In particular, we have $$
\pi'_!\bar{\bQ}_{l\widetilde{Y}'}=\bigoplus_{\rho\in
\Hat{W}}(\rho\otimes(\pi'_!\bar{\bQ}_{l\widetilde{Y}'})_{\rho}),$$
where $(\pi'_!\bar{\bQ}_{l\widetilde{Y}'})_{\rho}$ is an irreducible
local system on $Y$. We have $$
\varphi'_!\bar{\bQ}_{lX'}=\bigoplus_{\rho\in
\Hat{W}}(\rho\otimes(\varphi'_!\bar{\bQ}_{lX'})_{\rho}), $$ where
$(\varphi'_!\bar{\bQ}_{lX'})_{\rho}=IC(G^*,(\pi'_!\bar{\bQ}_{l\widetilde{Y}'})_{\rho})$.
Moreover, for $a\in \A$, we have $\cT_a=\bigoplus_{\rho\in
\Hat{W}}(\rho\otimes(\cT_a)_{\rho})$. Set
$$\Lg^{*\varsigma}=\{\xi\in\Lg^*|\sigma(\xi)=\varsigma\}, X'^\varsigma=\varphi'^{-1}(\Lg^{*\varsigma})\subset X'.$$ We have
$\Lg^{*\varsigma}=\{\xi\in\Lg^*|\xi \text{ nilpotent}\}$. Let
$\varphi'^\varsigma:X'^\varsigma\rightarrow\Lg^{*\varsigma}$ be the
restriction of $\varphi':X'\rightarrow\Lg^*$.
\begin{lemma}\label{lem-reg}
There exists a nilpotent element $\xi$ in $\Lg^*$ such that the set
$\{B_1\in\cB|\xi\in\Ln_1'\}$ is finite.
\end{lemma}
\begin{proof}
Let $R$ be the root system of $G$ relative to $T$. We have a weight
space decomposition $\Lg=\Lt\oplus\oplus_{\alpha\in R}\Lg_{\alpha}$,
where $\Lg_{\alpha}=\{x\in\Lg|Ad(t)x=\alpha(t)x,\forall t\in T\}$ is
one dimensional for $\alpha\in R$ (see for example \cite{Sp2}). Let
$\alpha_i,i=1,\ldots,r$ be a set of simple roots in $R$ such that
$\Lb=\Lt\oplus\oplus_{\alpha\in R^+}\Lg_{\alpha}$ and
$x_\alpha,\alpha\in R,h_{\alpha_i}$ be a Chevalley basis in $\Lg$.
Let $x_\alpha^*$ and $h_{\alpha_i}^*$ be the dual basis in $\Lg^*$.
Set $\xi=\sum_{i=1}^r x_{-\alpha_i}^*$. Then $\xi\in\Ln'$.

We show that $\{B_1\in\cB|\xi\in\Ln_1'\}=\{B\}$. Assume
$g.\xi\in\Ln'$. We have $\xi(g^{-1}\Lb g)=0$. By Bruhat
decomposition, we can write $g^{-1}=vn_wb$, where $v\in U\cap
wUw^{-1}$ and $n_w\in NT$ is a representative for $w\in W$. Assume
$w\neq 1$. There exists $1\leq i\leq r$ such that
$w^{-1}\alpha_i<0$. Let $\alpha=-w^{-1}\alpha_i>0$. We have
$\xi(\Ad(vn_w)x_\alpha)=\xi(c\Ad(v)x_{-\alpha_i})=\xi(cx_{-\alpha_i})=c$,
where $c$ is a nonzero constant. This contradicts $\xi(g^{-1}\Lb
g)=0$. Thus $w=1$ and $g^{-1}.\Ln'=\Ln'$.
\end{proof}
\begin{lemma}\label{l-1}
$\mathrm{(i)}$ $X'^\varsigma$ and $\Lg^{*\varsigma}$ are irreducible
varieties of dimension $d_0=\dim G-r$.

$\mathrm{(ii)}$ We have
$(\varphi'_!\bar{\bQ}_{lX'})|_{\Lg^{*\varsigma}}=\varphi'^\varsigma_!\bar{\bQ}_{lX'^\varsigma}$.
Moreover, $\varphi'^\varsigma_!\bar{\bQ}_{lX'^\varsigma}[d_0]$ is a
semisimple perverse sheaf on $\Lg^{*\varsigma}$.

$\mathrm{(iii)}$ We have
$(\varphi'_!\bar{\bQ}_{lX'})_\rho|_{\Lg^{*\varsigma}}\neq 0$ for any
$\rho\in \Hat{W}$.
\end{lemma}
\begin{proof}
(i) We have $X'^\varsigma
=\{(\xi,gB)\in\Lg^*\times\cB|g^{-1}.\xi\in\Ln\}$. By projection to
the second coordinate, we see that $\dim
X'^\varsigma=\dim\Ln'+\dim\cB=\dim G-r$. The map
$\varphi'^\varsigma$ is surjective and the fiber at some point $\xi$
is finite (see Lemma \ref{lem-reg}). It follows that
$\dim\Lg^{*\varsigma}=\dim G-r$. This proves (i).

The first assertion of (ii) follows from base change theorem. Since
$\varphi'^\varsigma$ is proper, by similar argument as in the proof
of Proposition \ref{p-2}, to show that
$\varphi'^\varsigma_!\bar{\bQ}_{lX'^\varsigma}[d_0]$ is a perverse
sheaf, it suffices to show $$\forall\ i\geq
0,\dim\supp\cH^i(\varphi'^\varsigma_!\bar{\bQ}_{lX'^\varsigma})\leq
\dim\Lg^{*\varsigma}-i.$$ It is enough to show $\forall\ i\geq 0$,
$\dim\{\xi\in\Lg^{*\varsigma}|\dim(\varphi'^\varsigma)^{-1}(\xi)\geq
i/2\}\leq\dim\Lg^{*\varsigma}-i$. If this is not true for some
$i\geq 0$, it would follow that the variety $\{(\xi,B_1,B_2)\in
\mathfrak{g}^*\times\cB\times\cB|\xi\in\mathfrak{n}_1'\cap\mathfrak{n}_2'\}$
has dimension greater than $\dim\Lg^{*\varsigma}=\dim G-r$, which
contradicts to Lemma \ref{prop-dim}. This proves that
$\varphi'^\varsigma_!\bar{\bQ}_{lX'^\varsigma}[d_0]$ is a perverse
sheaf. It is semisimple by the decomposition theorem\cite{BBD}. This
proves (ii).

Now we prove (iii). By Lemma \ref{lc-1}, we have
$\cT_\varsigma^{\cO_1}=H^0_c(\Lt'\cap\sigma^{-1}(\varsigma),\bar{\bQ}_{l})\neq
0$. From Lemma \ref{l-multi}, we see that the $W$-module structure
defines an injective map
$\bar{\bQ}_l[W]\otimes\cT_\varsigma^{\cO_1}\rightarrow
\cT_\varsigma$. Since $\cT_\varsigma^{\cO_1}\neq 0$, we have
$(\bar{\bQ}_l[W]\otimes\cT_\varsigma^{\cO_1})_\rho\neq 0$ for any
$\rho\in \Hat{W}$, hence $(\cT_\varsigma)_\rho\neq 0$. We have
$\cT_\varsigma=H^{2d_0}_c(\Lg^{*\varsigma},\varphi'_!\bar{\bQ}_{lX'}\otimes\varphi'_!\bar{\bQ}_{lX'})$,
hence
$$\bigoplus_{\rho\in \Hat{W}}\rho\otimes(\cT_\varsigma)_\rho=\bigoplus_{\rho\in \Hat{W}}\rho\otimes
H^{2d_0}_c(\Lg^{*\varsigma},(\varphi'_!\bar{\bQ}_{lX'})_\rho\otimes\varphi'_!\bar{\bQ}_{lX'}).$$
This implies that
$(\cT_\varsigma)_{\rho}=H^{2d_0}_c(\Lg^{*\varsigma},(\varphi'_!\bar{\bQ}_{lX'})_\rho\otimes\varphi'_!\bar{\bQ}_{lX'})$.
Thus it follows from $(\cT_\varsigma)_{\rho}\neq 0$ that
$(\varphi'_!\bar{\bQ}_{lX'})_\rho|_{\Lg^{*\varsigma}}\neq 0$ for any
$\rho\in \Hat{W}$.
\end{proof}

Let $\mathfrak{A}'$ be the set of all pairs
$(\mathrm{c}',\mathcal{F}')$ where $\mathrm{c}'$ is a nilpotent
$G$-orbit in $\Lg^*$ and $\mathcal{F}'$ is an irreducible
$G$-equivariant local system on $\mathrm{c}'$ (up to isomorphism).

\begin{proposition}\label{mp-1}

$\mathrm{(i)}$ The restriction map
$\End_{\mathcal{D}(\Lg^*)}(\varphi'_!\bar{\bQ}_{lX'})
\rightarrow\End_{\mathcal{D}(\Lg^{*\varsigma})}(\varphi'^\varsigma_!\bar{\bQ}_{lX'^\varsigma})$
is an isomorphism.

$\mathrm{(ii)}$ For any $\rho\in \Hat{W}$, there is a unique
$(\mathrm{c}',\mathcal{F}')\in\mathfrak{A}'$ such that
$(\varphi'_!\bar{\bQ}_{lX'})_\rho|_{\Lg^{*\varsigma}}[d_0]$ is
$IC(\bar{\rc'},\cF')[\dim\rc']$ regarded as a simple perverse sheaf
on $\Lg^{*\varsigma}$ (zero outside $\bar{\rc'}$). Moreover,
$\rho\mapsto (\rc',\cF')$ is an injective map $\gamma:
\Hat{W}\rightarrow\mathfrak{A}'$.
\end{proposition}
\begin{proof}
(i). Recall that we have
$\varphi'_!\bar{\bQ}_{lX'}=\bigoplus_{\rho\in
\Hat{W}}\rho\otimes(\varphi'_!\bar{\bQ}_{lX'})_\rho$ where
$(\varphi'_!\bar{\bQ}_{lX'})_\rho[\dim\Lg^*]$ are simple perverse
sheaves on $\Lg^*$. Thus we have
$\varphi'_!\bar{\bQ}_{lX'}|_{\Lg^{*\varsigma}}=\varphi'^\varsigma_!\bar{\bQ}_{lX'^\varsigma}=\bigoplus_{\rho\in
\Hat{W}}\rho\otimes(\varphi'_!\bar{\bQ}_{lX'})_\rho|_{\Lg^{*\varsigma}}$
(we use Lemma \ref{l-1} (ii)). The restriction map
$\End_{\mathcal{D}(\Lg^*)}(\varphi'_!\bar{\bQ}_{lX'})\rightarrow
\End_{\mathcal{D}(\Lg^{*\varsigma})}(\varphi'^\varsigma_!\bar{\bQ}_{lX'^\varsigma})$
is factorized as
\begin{eqnarray*}
&&\bigoplus_{\rho\in
\Hat{W}}\End_{\mathcal{D}(\Lg^*)}(\rho\otimes(\varphi'_!\bar{\bQ}_{lX'})_\rho)\xrightarrow{b}\bigoplus_{\rho\in
\Hat{W}}\End_{\mathcal{D}(\Lg^{*\varsigma})}(\rho\otimes(\varphi'_!\bar{\bQ}_{lX'})_\rho|_{\Lg^{*\varsigma}})
\xrightarrow{c}\End_{\mathcal{D}(\Lg^{*\varsigma})}(\varphi'^\varsigma_!\bar{\bQ}_{lX'^\varsigma})
\end{eqnarray*}
where $b=\oplus_\rho b_\rho$,
$b_\rho:\End(\rho)\otimes\End_{\mathcal{D}(\Lg^*)}((\varphi'_!\bar{\bQ}_{lX'})_\rho)
\rightarrow\End(\rho)\otimes\End_{\mathcal{D}(\Lg^{*\varsigma})}((\varphi'_!\bar{\bQ}_{lX'})_\rho|_{\Lg^{*\varsigma}}).$
By Lemma \ref{l-1} (iii),
$(\varphi'_!\bar{\bQ}_{lX'})_\rho|_{\Lg^{*\varsigma}}\neq 0$, thus
$\End_{\mathcal{D}(\Lg^*)}((\varphi'_!\bar{\bQ}_{lX'})_\rho)=\bar{\bQ}_l
\subset\End_{\mathcal{D}(\Lg^{*\varsigma})}((\varphi'_!\bar{\bQ}_{lX'})_\rho|_{\Lg^{*\varsigma}})$.
It follows that $b_\rho$ and thus $b$ is injective. Since $c$ is
also injective, the restriction map is injective. Hence it remains
to show that
$$\dim\End_{\mathcal{D}(\Lg^{*\varsigma})}(\varphi'^\varsigma_!\bar{\bQ}_{lX'^\varsigma})=\dim\End_{\mathcal{D}(\Lg^*)}(\varphi'_!\bar{\bQ}_{lX'}).$$
For $A,A'$ two simple perverse sheaves on a variety $X$, we have
$H_c^0(X,A\otimes A')=0$ if and only if $A$ is not isomorphic to
$\mathfrak{D}(A')$ and $\dim H_c^0(X,A\otimes \mathfrak{D}(A))=1$
(see \cite{Lu2} section 7.4). We apply this to the semisimple
perverse sheaf $\varphi'^\varsigma_!\bar{\bQ}_{lX'^\varsigma}[d_0]$
on $\Lg^{*\varsigma}$ and get
\begin{eqnarray*}
&&\dim\End_{\mathcal{D}(\Lg^{*\varsigma})}(\varphi'^\varsigma_!\bar{\bQ}_{lX'^\varsigma})=\dim
H^0_c(\Lg^{*\varsigma},\varphi'^\varsigma_!\bar{\bQ}_{lX'^\varsigma}[d_0]\otimes\mathfrak{D}(\varphi'^\varsigma_!\bar{\bQ}_{lX'^\varsigma}[d_0]))\\&&=\dim
H^0_c(\Lg^{*\varsigma},\varphi'^\varsigma_!\bar{\bQ}_{lX'^\varsigma}[d_0]\otimes\varphi'^\varsigma_!\bar{\bQ}_{lX'^\varsigma}[d_0])
=\dim
H^{2d_0}_c(\Lg^{*\varsigma},\varphi'^\varsigma_!\bar{\bQ}_{lX'^\varsigma}\otimes\varphi'^\varsigma_!\bar{\bQ}_{lX'^\varsigma})\\&&=\dim
H^{2d_0}_c(\Lg^{*\varsigma},\varphi'_!\bar{\bQ}_{lX'}\otimes\varphi'_!\bar{\bQ}_{lX'})=\dim\cT_\varsigma=\sum_{w\in
W}\dim \cT_\varsigma^{\cO_w}.
\end{eqnarray*}
(The fourth equality follows from Lemma \ref{l-1} (ii) and the last
one follows from Lemma \ref{p-1}.)

We have
${}\cT_\varsigma^{\cO_w}=H^0_c(\bar{\sigma}^{-1}(\varsigma),\bar{\bQ}_l)$
(see Lemma \ref{lc-1}), hence $\dim {}\cT_\varsigma^{\cO_w}=1$ and
$$\sum_{w\in W}\dim \cT_\varsigma^{\cO_w}=|W|=\dim\End_{\mathcal{D}(\Lg^*)}(\varphi'_!\bar{\bQ}_{lX'}).$$
Thus (i) is proved.

From the proof of (i) we see that both $b$ and $c$ are isomorphisms.
It follows that the perverse sheaf
$(\varphi'_!\bar{\bQ}_{lX'})_\rho|_{\Lg^{*\varsigma}}[d_0]$ on
$\Lg^{*\varsigma}$ is simple and that for $\rho,\rho'\in \Hat{W}$,
we have
$(\varphi'_!\bar{\bQ}_{lX'})_\rho|_{\Lg^{*\varsigma}}[d_0]\cong(\varphi'_!\bar{\bQ}_{lX'})_{\rho'}|_{\Lg^{*\varsigma}}[d_0]$
if and only if $\rho=\rho'$. Since the simple perverse sheaf
$(\varphi'_!\bar{\bQ}_{lX'})_\rho|_{\Lg^{*\varsigma}}[d_0]$ is
$G$-equivariant and $\Lg^{*\varsigma}$ consists of finitely many
nilpotent $G$-orbits,
$(\varphi'_!\bar{\bQ}_{lX'})_\rho|_{\Lg^{*\varsigma}}[d_0]$ must be
as in (ii).
\end{proof}

\subsection{}
In this subsection let $G=SO_N(\tk)$ (resp. $Sp_{2n}(\tk)$) and
$\Lg=\Lo_N(\tk)$ (resp. $\mathfrak{sp}_{2n}(\tk)$) be the Lie
algebra of $G$. Let $G_{s}$ be a simply connected group over $\tk$
of the same type as $G$ and $\Lg_{s}$ be the Lie algebra of $G_{s}$.
For $q$ a power of 2, let $G(\tF_q)$, $\mathfrak{g}({\tF}_q)$ be the
fixed points of a split Frobenius map $\mathfrak{F}_q$ relative to
$\tF_q$ on $G$, $\Lg$. Let $G_{s}(\tF_q)$, $\Lg_{s}(\tF_q)$ be
defined like $G(\tF_q)$, $\Lg(\tF_q)$. Let $\mathfrak{A}'$ be the
set of all pairs $(\mathrm{c}',\mathcal{F}')$ where $\mathrm{c}'$ is
a nilpotent $G$-orbit in $\Lg^*$ and $\mathcal{F}'$ is an
irreducible $G$-equivariant local system on $\mathrm{c}'$ (up to
isomorphism). Let $\mathfrak{A}_{s}'$ be defined for $G_{s}$ as in
the introduction. We show that the number of elements in
$\mathfrak{A}_{s}'$ is equal to the number of elements in
$\mathfrak{A}'$.

We first show that the number of elements in $\mathfrak{A}'$ is
equal to the number of nilpotent $G(\tF_q)$-orbits in
$\mathfrak{g}({\tF}_q)^*$ (for $q$ large). To see this we can assume
$\tk=\bar{\tF}_2$.   Pick representatives $\xi_1,\ldots,\xi_M$ for
the nilpotent $G$-orbits in $\Lg^*$. If $q$ is large enough, the
Frobenius map $\mathfrak{F}_q$ keeps $\xi_i$ fixed and acts
trivially on $Z_{G}(\xi_i)/Z_{G}^0(\xi_i)$. Then the number of
$G(\tF_q)$-orbits in the $G$-orbit of $\xi_i$ is equal to the number
of irreducible representations of $Z_{G}(\xi_i)/Z_{G}^0(\xi_i)$
hence to the number of $G$-equivariant irreducible local systems on
the $G$-orbit of $\xi_i$. Similarly, the number of elements in
$\mathfrak{A}_{s}'$ is equal to the number of nilpotent
$G_{s}(\tF_q)$-orbits in $\mathfrak{g}_{s}({\tF}_q)^*$.

On the other hand, the number of nilpotent $G(\tF_q)$-orbits in
$\mathfrak{g}({\tF}_q)^*$ is equal to the number of nilpotent
$G_{s}(\tF_q)$-orbits in $\mathfrak{g}_{s}({\tF}_q)^*$. In fact, we
have a morphism $G_s\rightarrow G$ which is an isomorphism of
abstract groups and an obvious bijective morphism
$\mathcal{N}'\rightarrow \mathcal{N}'_{s}$ where $\mathcal{N}'$
(resp. $\mathcal{N}'_{s}$) is the set of nilpotent elements in $
\Lg^*$ (resp. $\Lg_{s}^*$). Thus the nilpotent orbits in $\Lg^*$ and
$\Lg_{s}^*$ are in bijection and the corresponding
 component groups of centralizers are isomorphic. It follows that
 $|\mathfrak{A}'|=|\mathfrak{A}_{s}'|$.

\begin{corollary}\label{coro-1}
$|\mathfrak{A}'|=|\mathfrak{A}_{s}'|=|\hat{W}|$.
\end{corollary}
\begin{proof}
Assume $G$ is $SO_{2n}(\tk)$. The assertion follows from the above
argument, Proposition \ref{prop-3}, and Corollary 6.17 in \cite{X}.
Assume $G$ is $Sp_{2n}(\tk)$ or $O_{2n+1}(\tk)$. It follows from
Proposition \ref{mp-1} (ii) that
$|\mathfrak{A}'|=|\mathfrak{A}_{s}'|$ is greater than $|\Hat{W}|$.
On the other hand, it is known that $|\hat{W}|=p_2(n)$ (see
\cite{Lu3}). Hence $|\mathfrak{A}'|=|\mathfrak{A}_{s}'|$ is less
than $|\Hat{W}|$ by Proposition \ref{prop-symp}, Proposition
\ref{prop-orth} and the above argument.
\end{proof}

\begin{theorem}\label{coro-3}
The map $\gamma$ in Proposition \ref{mp-1} $\mathrm{(ii)}$ is a
bijection.
\end{theorem}

\begin{corollary}\label{coro-2}
Proposition \ref{prop-1}, Corollary \ref{cor-1}, Proposition
\ref{prop-symp}, Proposition \ref{prop-2} and Proposition
\ref{prop-orth} hold with all "at most" removed.
\end{corollary}
\begin{proof}
For $q$ large enough, this follows from Corollary \ref{coro-1}. Now
let $q$ be an arbitrary power of $2$.  Let $(\rc',\cF')$ be a pair
in $\mathfrak{A}_{s}'$. Since the Springer correspondence map
$\gamma$ in Proposition \ref{mp-1} $\mathrm{(ii)}$ is bijective by
Corollary \ref{coro-1}, there exists $\rho\in\hat{W}$ corresponding
to $(\rc',\cF')$ under the map $\gamma$. It follows that the pair
$(\mathfrak{F}_q^{-1}(\rc'),\mathfrak{F}_q^{-1}(\cF'))$ corresponds
to $\mathfrak{F}_q^{-1}(\rho)\in\hat{W}$. Since the Frobenius map
$\mathfrak{F}_q$ acts trivially on $W$ and $\gamma$ is injective, it
follows that $\rc'$ is stable under $\mathfrak{F}_q$ and
$\mathfrak{F}_q^{-1}(\cF')\cong\cF'$. Pick a rational point $\xi$ in
$\rc'$. The $G_{s}$-equivariant local systems on $\rc$ are in 1-1
correspondence with the isomorphism classes of the irreducible
representations of $Z_{G_{s}}(\xi)/Z_{G_{s}}^0(\xi)$. Since
$Z_{G_{s}}(\xi)/Z_{G_{s}}^0(\xi)$ is abelian (see Proposition
\ref{prop-c1} and \ref{prop-c2}) and the Frobenius map
$\mathfrak{F}_q$ acts trivially on the irreducible representations
of $Z_{G_{s}}(\xi)/Z_{G_{s}}^0(\xi)$, $\mathfrak{F}_q$ acts
trivially on $Z_{G_{s}}(\xi)/Z_{G_{s}}^0(\xi)$. Thus it follows that
the number of nilpotent $G_{s}(\tF_q)$-orbits in
$\mathfrak{g}_{s}({\tF}_q)^*$ is independent of $q$ hence it is
equal to $|\mathfrak{A}_{s}'|=|\hat{W}|$.
\end{proof}
\begin{remark}Let $G_{ad}$ be an adjoint algebraic group of type $B,C$ or $D$
over $\tk$ and $\Lg_{ad}$ be its Lie algebra. Let $\Lg_{ad}^*$ be
the dual space of $\Lg_{ad}$. In \cite{X}, we have constructed a
Springer correspondence for $\Lg_{ad}$. One can construct a Springer
correspondence for $\Lg_{ad}^*$ using the result for $\Lg_{ad}$ and
the
Deligne-Fourier transform. 
We expect the two Springer correspondences coincide
(up to sign representation of the Weyl group). We use the approach
presented above since this construction is more suitable for
computing the explicit Springer correspondence.
\end{remark}
\section{centralizers and component groups}
\subsection{}
In this subsection assume $G=Sp(2N)$. We study some properties of
the centralizer $Z_G(\xi)$ for a nilpotent element $\xi\in\Lg^*$ and
the component group $Z_G(\xi)/Z_G^0(\xi)$. Let
$V={^*W}_{\chi(m_1)}(m_1)\oplus
{^*W}_{\chi(m_2)}(m_2)\oplus\cdots\oplus {^*W}_{\chi(m_s)}(m_s)$,
$m_1\geq\cdots\geq m_s$, be a form module corresponding to
$\xi\in\Lg^*$. Let $T_\xi$ be defined as in subsection
\ref{ssec-1-1}. We have $Z_G(\xi)=Z(V)=\{g\in
GL(V)|\beta(gv,gw)=\beta(v,w),\alpha_\xi(gv)=\alpha_\xi(g), \
\forall\ v,w\in V\}$.
\begin{proposition} $\dim
Z(V)=\sum_{i=1}^{s}((4i-1)m_i-2\chi(m_i))$.
\end{proposition}
\begin{proof}
We argue by induction on $s$. The case $s=1$ can be easily verified.
Let $C(V)=\{g\in GL(V)\ |\ gT_\xi=T_\xi g\}$. Let
$V_1={^*W}_{\chi(m_1)}(m_1)$ and
$V_2={^*W}_{\chi(m_2)}(m_2)\oplus\cdots\oplus
{^*W}_{\chi(m_s)}(m_s)$. We consider $V_1$ as an element in the
Grassmannian variety $Gr(V,2m_1)$ and consider the action of $C(V)$
on $Gr(V,2m_1)$. Then the orbit of $V_1$ has dimension
$\dim\Hom_A(V_1,V_2)=4\sum_{i=2}^sm_i$. Now we consider the action
of $Z(V)$ on $Gr(V,2m_1)$. The orbit $Z(V)V_1$ is open dense in
$C(V)V_1$ and thus has dimension $4\sum_{i=2}^sm_i$. We claim that
$$(*)\text{ the stabilizer of }V_1\text{ in }Z(V)\text{ is the product of
}Z(V_1)\text{ and }Z(V_2).$$ Thus using induction hypothesis and
$(*)$ we get $\dim Z(V)=\dim Z(V_1)+\dim Z(V_2)+\dim
Z(V)V_1=3m_1-2\chi(m_1)+\sum_{i=2}^s((4i-5)m_i-2\chi(m_i))+4\sum_{i=2}^sm_i=\sum_{i=1}^{s}((4i-1)m_i-2\chi(m_i))$.

Proof of $(*)$: Assume $g:V_1\oplus V_2\rightarrow V_1\oplus V_2$
lies in the stabilizer of $V_1$ in $Z(V)$. Let $p_{ij}$, $i,j=1,2$
be the obvious projection composed with $g$. Then $p_{12}=0$. We
claim that $p_{11}$ is non-singular. It is enough to show that
$p_{11}$ is injective. Assume $p_{11}(v_1)=0$ for some $v_1\in V_1$.
Then we have
$\beta(gv_1,gv_1')=\beta(p_{11}v_1,gv_1')=0=\beta(v_1,v_1')$ for any
$v_1'\in V_1$. Since $\beta|_{V_1}$ is non-degenerate, we get
$v_1=0$. Now for any $v_2\in V_2,v_1\in V_1$, we have
$\beta(gv_1,gv_2)=\beta(p_{11}v_1,p_{21}v_2+p_{22}v_2)=\beta(p_{11}v_1,p_{21}v_2)=\beta(v_1,v_2)=0$.
Since $\beta|_{V_1}$ is non-degenerate and $p_{11}$ is bijective on
$V_1$, we get $p_{21}(v_2)=0$. Then $(*)$ follows.
\end{proof}

Let $r=\#\{1\leq i\leq s|\chi(m_i)+\chi(m_{i+1})<m_i \text{ and
}\chi(m_i)>\frac{m_i-1}{2}\}$.
\begin{proposition}\label{prop-c1}
The component group $Z(V)/Z^0(V)$ is $(\mathbb{Z}/2\mathbb{Z})^r$.
\end{proposition}
\begin{proof}

Assume $q$ large enough. By the same argument as in the proof of
Proposition 7.1 in \cite{X} one shows that $Z(V)/Z^0(V)$ is an
abelian group of order $2^{r}$. We show that there is a subgroup
$(\mathbb{Z}/2\mathbb{Z})^{r}\subset Z(V)/Z(V)^0$. Thus
$Z(V)/Z(V)^0$ has to be $(\mathbb{Z}/2\mathbb{Z})^{r}$. Let $1\leq
i_1,\ldots,i_{r}\leq s$ be such that $\chi(m_{i_j})>(m_{i_j}-1)/2$
and $\chi(m_{i_j})+\chi(m_{i_{j}+1})< m_{i_j}$, $j=1,\ldots,r$.

Let
$V_j={^*W}_{\chi(m_{i_{j-1}+1})}(m_{i_{j-1}+1})\oplus\cdots\oplus
{^*W}_{\chi(m_{i_{j}})}(m_{i_{j}})$, $j=1,\ldots,r-1$, where
$i_0=0$, and
$V_r={^*W}_{\chi(m_{i_{r-1}+1})}(m_{i_{r-1}+1})\oplus\cdots\oplus
{^*W}_{\chi(m_{s})}(m_{s})$. Then $V=V_1\oplus V_2\oplus\cdots\oplus
V_{r}$. We have $Z(V_i)/Z^0(V_i)=\mathbb{Z}/2\mathbb{Z}$,
$i=1,\ldots,r$. Take $g_i\in Z(V_i)$ such that $g_iZ^0(V_i)$
generates $Z(V_i)/Z^0(V_i)$, $i=1,\ldots,r$. Let
$\tilde{g_i}=Id\oplus\cdots\oplus g_i\oplus\cdots\oplus Id$,
$i=1,\ldots,r$. Then we have $\tilde{g_i}\in Z(V)$ and
$\tilde{g_i}\notin Z^0(V)$. We also have that the images of
$\tilde{g}_{i_1}\cdots\tilde{g}_{i_p}$'s, $1\leq i_1<\cdots<i_p\leq
r$, $p=1,\ldots,r$, in $Z(V)/Z^0(V)$ are not equal to each other.
Moreover $\tilde{g}_i^2\in Z^0(V)$. Thus the $\tilde{g}_iZ^0(V)$'s
generate a subgroup $(\mathbb{Z}/2\mathbb{Z})^{r}$ in $Z(V)/Z^0(V)$.
\end{proof}

\subsection{}
In this subsection assume $G=O(2N+1)$. We study some properties of
the centralizer $Z_G(\xi)$ for a nilpotent element $\xi\in\Lg^*$ and
the component group $Z_G(\xi)/Z_G^0(\xi)$. Let
$(V,\alpha,\beta_\xi)$ be a form module corresponding to
$\xi\in\Lg^*$. Assume the corresponding pair of partitions is
$(\nu_0,\ldots,\nu_s)(\mu_1,\ldots,\mu_s)$. Let $C(V)=\{g\in
GL(V)|\beta(gv,gw)=\beta(v,w), \beta_\xi(gv,gw)=\beta_\xi(v,w),\
\forall\ v,w\in V\}$. We have $Z_G(\xi)=Z(V)=\{g\in
C(V)|\alpha(gv)=\alpha(v), \ \forall\ v\in V\}.$
\begin{lemma}
$|Z(V_{2m+1})(\tF_q)|=q^m$ and $|C(V_{2m+1})(\tF_q)|=q^{2m+1}.$
\end{lemma}
\begin{proof}
Let $V_{2m+1}=\text{span}\{v_0,\cdots,v_m,u_0,\cdots,u_{m-1}\}$,
where $v_i,u_i$ are chosen as in Lemma \ref{lem-n-3} and Lemma
\ref{lem-vu}. Let $g\in C(V_{2m+1})$. Then $g: V_{2m+1}\rightarrow
V_{2m+1}$ satisfies $\beta(gv,gw)=\beta(v,w)$ and
$\beta_\xi(gv,gw)=\beta_\xi(v,w)$ for all $v,w\in V_{2m+1}$. Since
$\beta_\xi(\sum_{i=0}^{m}v_i\lambda^i,v)+\lambda\beta(\sum_{i=0}^{m}v_i\lambda^i,v)=0$,
we have
$\beta_\xi(\sum_{i=0}^{m}gv_i\lambda^i,v)+\lambda\beta(\sum_{i=0}^{m}gv_i\lambda^i,v)=0$
for all $v\in V_{2m+1}$. This implies
$\sum_{i=0}^{m}gv_i\lambda^i=a\sum_{i=0}^{m}v_i\lambda^i$ for some
$a\in\tF_q^*$. Namely we have $gv_i=av_i$, $i=0,\ldots,m$. Assume
$gu_i=\sum_{k=0}^{m}a_{ik}v_k+\sum_{k=0}^{m-1}b_{ik}u_k$. We have
$\beta(gv_i,gu_j)=ab_{ji}=\beta(v_i,u_j)=\delta_{i,j}$,
$\beta(gu_i,gu_j)=\sum_{k=0}^{m-1}a(a_{ik}b_{jk}+b_{ik}a_{jk})=a(a_{ij}+a_{ji})=\beta(u_i,u_j)=0$,
$\beta_\xi(gv_{i+1},gu_j)=ab_{ji}=\beta_\xi(v_{i+1},u_j)=\delta_{i,j}$
and
$\beta_\xi(gu_i,gu_j)=\sum_{k=1}^{m}a(a_{ik}b_{j,k-1}+b_{i,k-1}a_{jk})=a(a_{i,j+1}+a_{j,i+1})=\beta(u_i,u_j)=0$.
Thus we get $gv_i=av_i$,$i=0,\ldots,m$, and
$gu_i=u_i/a+\sum_{j=0}^{m} a_{ij}v_j$, $i=0,\ldots,m-1$, where
$a_{ij}=a_{ji},\ a_{i,j+1}=a_{j,i+1},\ 0\leq i,j\leq m-1.$ Hence
$|C(V_{2m+1})(\tF_q)|=q^{2m+1}$.

Now assume $g\in Z(V_{2m+1})$. Then we have additional conditions
$\alpha(gv_i)=a^2\alpha(v_i)=\alpha(v_i)=\delta_{i,m}\Rightarrow
a=1$ and $\alpha(gu_i)=\alpha(u_i/a+\sum_{j=0}^{m}
a_{ij}v_j)=a_{im}^2+a_{ii}/a=\alpha(u_i)=0$, $i=0,\ldots,m-1$. Hence
$|Z(V_{2m+1})(\tF_q)|=q^{m}$.
\end{proof}

Write $V=V_{2m+1}\oplus W$ as in Lemma \ref{lem-8}.
\begin{lemma}\label{lem-cd1}
$|C(V)(\tF_q)|=|C(V_{2m+1})(\tF_q)|\cdot|C(W)(\tF_q)|\cdot q^{\dim
W}$.
\end{lemma}
\begin{proof}
Let $g\in C(V)$. Let $p_{11}:V_{2m+1}\rightarrow
V_{2m+1},p_{12}:V_{2m+1}\rightarrow W,\ p_{21}:W\rightarrow
V_{2m+1}$ and $p_{22}:W\rightarrow W$ be the projections composed
with $g$. Let $v_i,u_i$ be a basis of $V_{2m+1}$ as before. By the
same argument as in Lemma \ref{lem-e1}, we have $gv_i=av_i$ for some
$a$ and $p_{22}\in C(W)$. Let $w\in W$. Assume
$gu_i=\sum_{j=0}^{m}a_{ij}v_j+\sum_{j=0}^{m-1}b_{ij}u_j+p_{12}(u_i).$
We have $\beta(v_i,u_j)=\beta(gv_i,gu_j)=a b_{ji}=\delta_{i,j},$
thus $b_{ji}=\delta_{i,j}/a$. Now $\beta(gv_i,gw)=\beta(v_i,w)=0$
implies $p_{21}(w)=\sum_{i=0}^{m}b_i^wv_i$. Thus
$\beta(gu_i,gw)=\beta(p_{12}(u_i),p_{22}(w))+b_i^w/a=0$ and
$\beta_\xi(gu_i,gw)=\beta_\xi(p_{12}(u_i),p_{22}(w))+b_{i+1}^w/a=0$.
This gives us $b_i^w=a\beta(p_{12}(u_i),p_{22}(w)),\ i=0,\ldots,m-1,
b_i^w=a\beta_\xi(p_{12}(u_{i-1}),p_{22}(w)),\ i=1,\ldots,m. $ Thus
$\beta_\xi(p_{12}(u_{i-1}),p_{22}(w))=\beta(p_{12}(u_i),p_{22}(w))$,
$ i=1,\ldots,m-1$. This holds for any $w\in W$. Recall that on $W$,
we have $\beta_\xi(p_{12}(u_{i-1}),p_{22}(w))=\beta(T_\xi
p_{12}(u_{i-1}),p_{22}(w))$. Since $p_{22}$ is nonsingular and
$\beta|_{W\times W}$ is nondegenerate, we get
$p_{12}(u_i)=T_\xi^ip_{12}(u_0)$, $i=0,\ldots,m-1$, and
$b_i^w=a\beta(T_\xi^ip_{12}(u_0),p_{22}w).$

Hence we get $gv_i=av_i,i=0,\ldots,m,
gu_i=\sum_{j=0}^{m}a_{ij}v_j+u_i/a+T_\xi^ip_{12}(u_0),\
i=0,\ldots,m-1,
gw=\sum_{i=0}^{m}a\beta(T_\xi^ip_{12}(u_0),p_{22}w)v_i+p_{22}(w),\
\forall\ w\in W. $ Now note that $p_{12}(u_0)$ can be any vector in
$W$. It is easily verified that the lemma holds.
\end{proof}

\begin{proposition}
$\dim Z(V)=\nu_0+\sum_{i=1}^s\nu_i(4i+1)+\sum_{i=1}^s\mu_i(4i-1)$.
\end{proposition}
\begin{proof}
Let $V=V_{2m+1}\oplus W_{l_1}(m_1)\oplus\cdots\oplus
W_{l_s}(m_s)=(V,\alpha,\beta_\xi)$. Let
$W=W_{l_1}(m_1)\oplus\cdots\oplus W_{l_s}(m_s)$. We have $\dim
C(W)=\sum_{i=1}^s(4i-1)m_i$ and $\dim C(V_{2m+1})=2m+1$. By Lemma
\ref{lem-cd1}, $\dim C(V)=\dim C(W)+\dim V_{2m+1}+\dim
W=\sum_{i=1}^s(4i-1)m_i+2m+1+2\sum_{i=1}^sm_i$. Consider $V_{2m+1}$
as an element in the Grassmannian variety $Gr(V,2m+1)$. Let
$C(V)V_{2m+1}$ be the orbit of $V_{2m+1}$ under the action of
$C(V)$. The stabilizer of $V_{2m+1}$ in $C(V)$ is the product of
$C(V_{2m+1})$ and $C(W)$. Hence $\dim C(V)V_{2m+1}=\dim C(V)-\dim
C(V_{2m+1})-\dim C(W)=2\sum_{i=1}^sm_i$. We have $\dim
Z(V)V_{2m+1}=\dim C(V)V_{2m+1}$. Hence $\dim Z(V)=\dim
Z(V_{2m+1})+\dim Z(W)+\dim Z(V)V_{2m+1}
=m+\sum_{i=1}^s((4i+1)m_i-2l_i)=\nu_0+\sum_{i=1}^s\nu_i(4i+1)
+\sum_{i=1}^s\mu_i(4i-1)$.
\end{proof}

\begin{lemma}\label{lem-c2}
$|Z(V)(\tF_q)|=2^{k}q^{\dim Z(V)}+$ lower terms, where $k=\#\{i\geq
1|\nu_i<\mu_i\leq\nu_{i-1}\}$.
\end{lemma}
\begin{proof}
If $\#\{i\geq 1|\nu_i<\mu_i\leq\nu_{i-1}\}$=0, the assertion follows
from the classification of nilpotent orbits. Assume $1\leq t\leq s$
is the minimal integer such that $\nu_t<\mu_t\leq \nu_{t-1}$. Let
$V_1=V_{2m+1}\oplus W_1$ where $W_1=W_{l_1}(m_1)\oplus\cdots\oplus
W_{l_{t-1}}(m_{t-1})$ and $W_2=W_{l_t}(m_t)\oplus\cdots\oplus
W_{l_s}(m_s)$. We show that
\begin{equation}\label{eqn-1}
|Z(V)(\tF_q)|=|Z(V_1)(\tF_q)|\cdot|Z(W_2)(\tF_q)|\cdot q^{r_1},
\end{equation}
where $r_1=\dim W_2+\dim \Hom_A (W_1,W_2)$. We consider $V_1$ as an
element in the Grassmannian variety $Gr(V,\dim V_1)$. We have
\begin{eqnarray}\label{eqn-2}
|C(V)V_1(\tF_q)|&=&\frac{|C(V)(\tF_q)|}{|C(V_1)(\tF_q)|\cdot
|C(W_2)(\tF_q)|}\\
&=&\frac{|C(V_{2m+1})(\tF_q)|\cdot|C(W_1\oplus W_2)(\tF_q)|\cdot
q^{\dim(W_1+W_2)}} {|C(V_{2m+1})(\tF_q)|\cdot|C(W_1)(\tF_q)|\cdot
q^{\dim(W_1)}\cdot|C(W_2)(\tF_q)|}=q^{r_1}\nonumber. \end{eqnarray}
In fact, let $p_{ij},\ i,j=1,2,3$ be the projections of $g\in C(V)$.
Assume $g$ is in the stabilizer of $V_1$ in $C(V)$. Then we have
$p_{13}=p_{23}=0$. It follows from the same argument as in Lemma
\ref{lem-cd1} that $p_{22}$ is nonsingular and
$gv_i=av_i,i=0,\ldots,m,\ \
gu_i=\sum_{j=0}^{m}a_{ij}v_j+u_i/a+T_\xi^ip_{12}(u_0),\
i=0,\ldots,m-1,gw_1=\sum_{i=0}^{m}a\beta(T_\xi^ip_{12}(u_0),p_{22}w_1)v_i+p_{22}(w_1),\
\forall\ w_1\in
W_1,gw_2=\sum_{i=0}^{m}a\beta(T_\xi^ip_{12}(u_0),p_{22}w_2+p_{23}w_2)v_i+p_{22}(w_2)+p_{23}(w_2),\
\forall\ w_2\in W_2. $ Now
$\beta(gw_1,gw_2)=\beta(p_{22}(w_1),p_{22}(w_2)+p_{23}(w_2))=\beta(p_{22}(w_1),p_{22}(w_2))=0$,
for any $w_1\in W_1$ and $w_2\in W_2$. Since $p_{22}$ is nonsingular
and $\beta|_{W_1\times W_1}$ is nondegenerate, we get
$p_{22}(w_2)=0$ for any $w_2\in W_2$. Thus the stabilizer of $V_1$
in $C(V)$ is the product of $C(V_1)$ and $C(W_2)$ and (\ref{eqn-2})
follows.

We have $C(V)(V_1\oplus W_2)\cong C(V)(V_1)\oplus C(V)(W_2)$ implies
$C(V)(V_1)\cong V_1$ and $C(V)(W_2)\cong W_2$. Thus
$|C(V)(V_1)(\tF_q)|=|Z(V)V_1(\tF_q)|=q^{r_1}$. Since the stabilizer
of $V_1$ in $Z(V)$ is the product of $Z(V_1)$ and $Z(W_2)$,
(\ref{eqn-1}) follows. Now the lemma follows by induction hypothesis
since we have $\dim Z(V)=\dim Z(V_1)+\dim Z(W_2)+r_1$.
\end{proof}
\begin{proposition}\label{prop-c2}
The component group $Z(V)/Z^0(V)$ is $(\mathbb{Z}/2\mathbb{Z})^k$,
where $k=\#\{i\geq 1|\nu_i<\mu_i\leq\nu_{i-1}\}$.
\end{proposition}
\begin{proof}
Lemma \ref{lem-c2} and the classification of nilpotent orbits in
$\Lg(\tF_q)^*$($q$ large) show that $Z(V)/Z^0(V)$ is an abelian
group of order $2^k$. It is enough to show that there exists a
subgroup $(\mathbb{Z}/2\mathbb{Z})^k\subset Z(V)/Z^0(V)$. Assume
$V=V_{2m+1}\oplus W_{l_1}^{\epsilon_1}(m_1)\oplus\cdots\oplus
W_{l_s}^{\epsilon_s}(m_s)$. Let $i_1<i_2<\cdots<i_k$ be the $i$'s
such that $\nu_i<\mu_i\leq \nu_{i-1}$. Let $V_0=V_{2m+1}\oplus
W_{l_1}^{\epsilon_1}(m_1)\oplus\cdots\oplus
W_{l_{i_1-1}}^{\epsilon_{i_1-1}}(m_{i_1-1})$ and
$W_j=W_{l_{i_j}}^{\epsilon_{i_j}}(m_{i_j})\oplus\cdots\oplus
W_{l_{i_{j+1}-1}}^{\epsilon_{i_{j+1}-1}}(m_{i_{j+1}-1})$,
$j=1,\ldots,k$, where $i_{k+1}=s+1$. We have $Z(V_0)/Z^0(V_0)=\{1\}$
and $Z(W_j)/Z^0(W_j)=\mathbb{Z}/2\mathbb{Z}$, $j=1,\ldots,k$. Take
$g_j\in Z(W_j)$ such that $g_jZ^0(W_j)$ generates $Z(W_j)/Z^0(W_j)$.
Let $\tilde{g}_j=Id\oplus\cdots\oplus g_j\oplus\cdots\oplus Id$,
$j=1,\ldots,k$. Then $\tilde{g}_j\in Z(V)$, $\tilde{g}_j\notin
Z^0(V)$, $\tilde{g}_j^2\in Z^0(V)$ and
$\tilde{g}_{j_1}\tilde{g}_{j_2}\cdots\tilde{g}_{j_r}\notin Z^0(V)$
for any $1\leq j_1<j_2<\cdots<j_r\leq s$, $r=1,\ldots,k$. Thus
$\tilde{g_j}Z^0(W_j)$, $j=1,\ldots,k$ generate a subgroup
$\mathbb{Z}/2\mathbb{Z}^k$.
\end{proof}
\vskip 10pt {\noindent\bf\large Acknowledgement} \vskip 5pt I would
like to thank Professor George Lusztig for his guidance,
encouragement and many helpful discussions. I am also very grateful
to the referee for many valuable suggestions and comments.

\end{document}